\documentclass[12pt,oneside,english]{amsart}
\usepackage[T1]{fontenc}
\usepackage[latin9]{inputenc}
\usepackage{color}
\usepackage{babel}
\usepackage{verbatim}
\usepackage{amsbsy}
\usepackage{amstext}
\usepackage{amsthm}
\usepackage{amssymb}
\usepackage{esint}
\usepackage[unicode=true,pdfusetitle,
 bookmarks=true,bookmarksnumbered=false,bookmarksopen=false,
 breaklinks=false,pdfborder={0 0 1},backref=false,colorlinks=false]
 {hyperref}

\makeatletter
\numberwithin{equation}{section}
\numberwithin{figure}{section}
\theoremstyle{remark}
\newtheorem*{rem*}{\protect\remarkname}
\theoremstyle{plain}
\newtheorem{thm}{\protect\theoremname}[section]
\theoremstyle{plain}
\newtheorem{lem}[thm]{\protect\lemmaname}
\theoremstyle{plain}
\newtheorem*{lem*}{\protect\lemmaname}
\theoremstyle{plain}
\newtheorem{cor}[thm]{\protect\corollaryname}

\usepackage{babel}
\usepackage{graphics/graphicx}

\providecommand{\corollaryname}{Corollary}
\providecommand{\lemmaname}{Lemma}
\providecommand{\remarkname}{Remark}
\providecommand{\theoremname}{Theorem}

\makeatother

\providecommand{\corollaryname}{Corollary}
\providecommand{\lemmaname}{Lemma}
\providecommand{\remarkname}{Remark}
\providecommand{\theoremname}{Theorem}

\begin{document}
\newcommand\mpar[1]{\marginpar{\tiny \color{red} #1}} 
\newcommand\bblue[1]{{\color{blue} #1}} 
\newcommand\yyellow[1]{{\color{yellow} #1}} 
\newcommand\rred[1]{{\color{red} #1}} 
\newcommand\ggreen[1]{{\color{green} #1}}

\newcommand{\uu}{\mathbf{ u}}
\newcommand{\vv}{\mathbf{ v}}
\newcommand{\ww}{\mathbf{ w}}
\newcommand{\ee}{\mathbf{ e}}
\newcommand{\cL}{\mathcal{L}}
\def\avint{\mathop{\,\rlap{$\diagup$}\!\!\int}\nolimits}
\def\Xint#1{\mathchoice
{\XXint\displaystyle\textstyle{#1}}%
{\XXint\textstyle\scriptstyle{#1}}%
{\XXint\scriptstyle\scriptscriptstyle{#1}}%
{\XXint\scriptscriptstyle\scriptscriptstyle{#1}}%
\!\int}
\def\XXint#1#2#3{{\setbox0=\hbox{$#1{#2#3}{\int}$ }
\vcenter{\hbox{$#2#3$ }}\kern-.58\wd0}}
\def\ddashint{\Xint=}
\def\dashint{\Xint-}

\newcommand{\mres}{\mathbin{\vrule height 1.6ex depth 0pt width 0.13ex \vrule height 0.13ex depth 0pt width 1.3ex}}

\newcommand{\henrik}[1]{\begin{quotation}[\textbf{\color{red}Henrik's comment:\
		}{\color{red}\textit{#1}}]\end{quotation}}

	
	\date{}

\title{A weakly coupled  system of $p$-Laplace type in a heat conduction
 problem}
\author{Morteza Fotouhi, Mohammad Safdari, Henrik Shahgholian}
\maketitle

\begin{abstract}
We study temperature distribution in a heat conducting problem, for a system of p-Laplace equation, giving rise  to a free boundary.
\end{abstract}

 \tableofcontents

\section{Introduction}

\subsection{Background}

This paper delves into an extension of a classical optimization problem in heat conduction, presenting a succinct description as follows: Given a surface $\partial \Omega$ in $\mathbb{R}^{n}$ and positive constant functions defined on it (representing the temperature distribution), the aim is to enclose $\partial \Omega$ with a prescribed volume of insulating material to minimize heat loss in a stationary scenario. Mathematically, the objective is to discover a vector-valued function $\mathbf{u} = (u^1, \cdots , u^m)$ ($m \geq  1$) that corresponds to the temperature within $\Omega$. Whenever the components of $\mathbf{u}$ are nonnegative  and the volume of its support is equal to $1$, they become $p$-harmonic. The target is to minimize the heat flow, which can be regarded as a continuous family of convex functions dependent on $\nabla\mathbf{u}$ along $\partial \Omega$.

Our research was inspired by a series of  papers \cite{aguilera1986optimization,aguilera1987optimization,alt1981existence} and their generalization presented in \cite{teixeira2005nonlinear}. The initial two articles focused on studying constant temperature distributions, specifically in the linear case where $\Gamma(x,t)=t$. This linear setting enabled \cite{aguilera1986optimization,alt1981existence} to reduce the quantity to be minimized to the Dirichlet integral. However, even within the linear case, the problem of nonconstant temperature distribution, examined in \cite{aguilera1987optimization}, introduced various new challenges.

The main objective of our article is to explore the system version of the nonlinear case with a nonconstant temperature distribution, wherein the equation is governed by the $p$-Laplacian. The nonlinearity addressed in this paper  holds significant physical importance, as problems involving monotone operators, akin to those studied in \cite{teixeira2005nonlinear}, arise in the optimization of domains for electrostatic configurations.

The nonlinearity associated with $\nabla\mathbf{u}$ introduces various new challenges. For instance, computing normal derivatives of $W^{1,p}$-functions becomes problematic, leading to difficulties in providing a reasonable mathematical model. In \cite{aguilera1987optimization}, this challenge was overcome by minimizing the total mass of $\Delta u$, which can be treated as a nonnegative measure when $u$ is subharmonic. However, in the present case, there is no integral representation available for $\int_{\partial \Omega}\Gamma\big(x,A_{\nu}\mathbf{u}(x)\big),d\sigma$. To address this issue, similar to \cite{teixeira2005nonlinear}, we solve appropriate auxiliary variational  problems and compare them with the minimizer.

Now let us introduce the problem in mathematical framework. Let $\Omega\subset\mathbb{R}^{n}$ ($n\geq 2$) be a bounded open set with smooth boundary
whose volume $|\Omega|>1$. Consider the $p$-Laplace differential operator  ($1<p<\infty$) 
\[
\Delta_{p}u^{i}=\mathrm{div}\big(|\nabla u^{i}|^{p-2}\,\nabla u^{i}\big)=\mathrm{div}\big(A[u^{i}]\big),
\]
where we set $A[u^{i}]=A(\nabla u^{i}):=|\nabla u^{i}|^{p-2}\,\nabla u^{i}$
to simplify the notation.

Let $\boldsymbol{\varphi}:\partial \Omega\to \mathbb{R}^{m}$ be a $C^{1}$
function with positive components $\varphi^{i}>0$. For $\mathbf{u}:\Omega\to\mathbb{R}^{m}$
($m\geq 1$) satisfying 
\begin{equation}\label{eq: main_eq}
    \begin{cases}
\Delta_{p}u^{i}=0 & \textrm{in }\{|\mathbf{u}|>0\},\\
u^{i}=\varphi^{i} & \textrm{on }\partial \Omega,\\
\textrm{vol}(\mathrm{spt}\,|\mathbf{u}|)=1,
\end{cases}
\end{equation}
we want to minimize the functional 
\[
J(\mathbf{u}):=\int_{\partial \Omega}\Gamma\big(x,A_{\nu}u^{1}(x),\dots,A_{\nu}u^{m}(x)\big)\,d\sigma(x),
\]
where $\nu$ is the outward normal vector on $\partial \Omega$, 
\[
A_{\nu}u^{i}:=|\nabla u^{i}|^{p-2}\,\partial_{\nu}u^{i},
\]
and $\Gamma(x,\xi):\partial \Omega\times\mathbb{R}^{m}\to\mathbb{R}$ is
a continuous function that satisfies: 
\begin{enumerate}
\item For each fixed $x$, $\Gamma(x,\cdot)$ is a convex function. 
\item For every $i$, $\partial_{\xi_{i}}\Gamma(\cdot,\cdot)$ is positive
and has a positive lower bound on any set of the form $\{(x,\xi):\xi_{i}\ge a\}$.
In addition, $\partial_{\xi_{i}}\Gamma(\cdot,\cdot)$ is bounded above
on any set of the form $\{(x,\xi):\xi_{i}\le b\}$. (The bounds can
depend on $a,b$.) 
\item For each fixed $\xi$, $\partial_{\xi_{i}}\Gamma(\cdot,\xi)$ is a
$C^{1}$ function. 
\end{enumerate}
Note that, as a result, for every $\xi$ we have 
\begin{align*}
\Gamma(x,\xi_{1},\dots,\xi_{m}) & \ge\sum_{i\le m}\partial_{\xi_{i}}\Gamma(x,0)\xi_{i}+\Gamma(x,0)\\
 & \ge\sum_{i\le m}\psi_{i}(x)\xi_{i}-C,
\end{align*}
where $\psi_{i}(x):=\partial_{\xi_{i}}\Gamma(x,0)>0$ are positive $C^{1}$
functions and $C$ is a constant. In particular we have 
\begin{equation}
\Gamma(x,A_{\nu}u_{1},\dots,A_{\nu}u_{m})\ge\sum_{i=1}^{m}\psi_{i}(x)A_{\nu}u^{i}-C.\label{eq: Gamma coercive}
\end{equation}
A typical example of $\Gamma$ is 
\[
\Gamma(x,\xi)=\psi_{1}(x)\gamma_{1}(\xi_{1})+\dots+\psi_{m}(x)\gamma_{m}(\xi_{m}),
\]
where $\psi_{i}$'s are $C^{1}$ and positive, and $\gamma_{i}$'s
are $C^{1}$ increasing convex functions with positive derivative.

\subsection{Structure of the paper}

The structure of our paper is as follows: In Section 2, we introduce the physical problem under consideration. We then formulate a penalized version of the variational problem for the temperature $\mathbf{u}$ and define suitable constraint sets as part of our strategy to overcome the challenges arising from the nonlinearity. We solve the optimization problem over weakly closed subsets of $W^{1,p}$ (the sets $V_{\delta}$), establishing the optimal regularity properties of the minimizers, including Lipschitz regularity. These results are crucial for proving the existence of an optimal configuration for the original penalized problem, as discussed in Section 3. Here we also present fundamental geometric-measure properties of the optimal configuration, such as linear growth away from the free boundary and uniformly positive density. These properties allow us to establish a representation theorem following the framework of \cite{alt1981existence}.

In Section 4, we recover the original physical problem from the penalized problem by showing that for sufficiently small $\varepsilon$, the volume of $\{|\mathbf{u}_{\varepsilon}|>0\}$ automatically adjusts to be equal to $1$.

Section 5 is dedicated to the optimal regularity of the free boundary, for the case $p=2$. We demonstrate that the normal derivative of the minimizer along the free boundary is a H\"older continuous function, leading to the conclusion that the free boundary is a $C^{1,\alpha}$ surface. Furthermore, using the free boundary condition obtained during the proof of H\"older continuity, we establish that the free boundary is an analytic surface, except for a small singular set.

\section{The Penalized Problem }

Let $\Omega_{\delta}:=\{x\in \Omega:\textrm{dist}(x,\partial \Omega)<\delta\}$ and
\begin{align*}
V_{\delta}:=\{\mathbf{u}\in W^{1,p}(\Omega;\mathbb{R}^{m}):\; & u^{i}\ge0,\,\Delta_{p}u^{i}\ge0,\\
 & \Delta_{p}u^{i}=0\textrm{ in }\Omega_{\delta},\,u^{i}=\varphi^{i}\textrm{ on }\partial \Omega\}.
\end{align*}Furthermore, we set 
\[
V:=\bigcup_{\delta>0}V_{\delta}.
\]
Observe that the above definition is consistent due to the assumption
$\varphi^{i}>0$ on $\partial\Omega$. Also, by $\Delta_{p}u^{i}\ge0$
we mean that for any test function $\zeta\in C_{c}^{\infty}(\Omega)$
with $\zeta\ge0$ we have 
\[
-\int_{\Omega}\nabla\zeta\cdot|\nabla u^{i}|^{p-2}\nabla u^{i}\,dx\ge0.
\]
This implies that there is a Radon measure $\mu^{i}$ such that for any
test function $\zeta\in C_{c}^{\infty}(\Omega)$ we have  
\[
\int_{\Omega}\zeta\,d\mu^{i}=-\int_{\Omega}\nabla\zeta\cdot|\nabla u^{i}|^{p-2}\nabla u^{i}\,dx.
\]
To simplify the notation, we denote $\mu^{i}$ by $\Delta_{p}u^{i}$,
and $d\mu^{i}$ by $\Delta_{p}u^{i}\,dx$. (It should be noted that
this notation is not meant to imply $\mu^{i}$ is absolutely continuous
with respect to the Lebesgue measure. In fact, for the minimizer,
the two measures are mutually singular as we will see in Theorem \ref{Thm:Hausdorff-dim-FB}.)

Let $f_{\varepsilon}:\mathbb{R}\to\mathbb{R}$ be 
\[f_{\varepsilon}(t):=\begin{cases}
1+\frac{1}{\varepsilon}(t-1) & t\ge1,\\
1+\varepsilon(t-1) & t<1.
\end{cases}\]
We are interested in minimizing the penalized functional 
\[
J_{\varepsilon}(\mathbf{u}):=\int_{\partial \Omega}\Gamma\big(x,A_{\nu}\mathbf{u}(x)\big)\,d\sigma+f_{\varepsilon}\big(\big|\{|\mathbf{u}|>0\}\big|\big)
\]
over $V$. The significance of the above penalization is that it forces the volume $\big|\{|\mathbf{u}|>0\}\big|$ to be $1$ for small enough $\varepsilon$; see Theorem \ref{thm: vol is 1}. (Notice that the components of $\mathbf{u}\in V$ are $p$-harmonic
near $\partial \Omega$; therefore they are smooth enough near the boundary, and it makes sense to compute their derivatives along $\partial \Omega$.)
We first consider the minimizer of $J_{\varepsilon}$
over $V_{\delta}$. 
\begin{lem}
Let $\mathbf{u}\in V$. Then we have 
\begin{equation}
\int_{\Omega}\psi_{i}\Delta_{p}u^{i}\,dx+\int_{\Omega}\nabla\psi_{i}\cdot A[u^{i}]\,dx=\int_{\partial \Omega}\psi_{i}A_{\nu}u^{i}\,d\sigma,\label{eq: by parts}
\end{equation}
where $\psi_{i}$'s are $C^{1}$ functions. 

\end{lem}

\begin{proof}
Let $\phi_{k}\in C^{\infty}(\Omega)$ be such that $\phi_{k}\equiv1$ on
$\tilde{\Omega}_{k}:=\Omega-\Omega_{1/k}$ and $\phi_{k}\equiv0$ on $\partial \Omega$.
We know that $\mathbf{u}\in V_{\delta}$ for some $\delta$. Suppose
$k$ is large enough so that $1/k<\delta$, and thus $\Omega_{1/k}\subset \Omega_{\delta}$.
Then we have 
\begin{align*}
\int_{\Omega}\nabla(\phi_{k}\psi_{i})\cdot A[u^{i}]\,dx & =\int_{\Omega_{1/k}}\nabla(\phi_{k}\psi_{i})\cdot A[u^{i}]\,dx\\
 & \qquad\qquad+\int_{\tilde{\Omega}_{k}}\nabla(\underset{\equiv\,1}{\underset{\rotatebox[origin=c]{90}{\big\{}}{\phi_{k}}}\psi_{i})\cdot A[u^{i}]\,dx\\
 & =\int_{\Omega_{1/k}}\nabla(\phi_{k}\psi_{i})\cdot A[u^{i}]\,dx+\int_{\tilde{\Omega}_{k}}\nabla\psi_{i}\cdot A[u^{i}]\,dx.
\end{align*}
Now, noting that $\partial \Omega_{1/k}=\partial\tilde{\Omega}_{k}\cup\partial \Omega$,
and by using the integration by parts formula proved in \cite{anzellotti1983pairings},
we get 
\begin{align*}
\int_{\Omega_{1/k}}\nabla(\phi_{k}\psi_{i})\cdot A[u^{i}]\,dx & =\int_{\Omega_{1/k}}\nabla(\phi_{k}\psi_{i})\cdot A[u^{i}]+\phi_{k}\psi_{i}\negthickspace\negthickspace\underset{=\,0\textrm{ on }\Omega_{1/k}}{\underbrace{\Delta_{p}u^{i}}}\,dx\\
 & =-\int_{\partial\tilde{\Omega}_{k}}\underset{\equiv\,1}{\underset{\rotatebox[origin=c]{90}{\big\{}}{\phi_{k}}}\psi_{i}A_{\nu}u^{i}\,d\sigma+\int_{\partial \Omega}\underset{\equiv\,0}{\underset{\rotatebox[origin=c]{90}{\big\{}}{\phi_{k}}}\psi_{i}A_{\nu}u^{i}\,d\sigma\\
 & =-\int_{\partial\tilde{\Omega}_{k}}\psi_{i}A_{\nu}u^{i}\,d\sigma\underset{k\to\infty}{\longrightarrow}-\int_{\partial \Omega}\psi_{i}A_{\nu}u^{i}\,d\sigma.
\end{align*}
In addition, we have $\int_{\tilde{\Omega}_{k}}\nabla\psi_{i}\cdot A[u^{i}]\,dx\underset{k\to\infty}{\longrightarrow}\int_{\Omega}\nabla\psi_{i}\cdot A[u^{i}]\,dx$,
and 
\begin{align*}
\int_{\Omega}\nabla(\phi_{k}\psi_{i})\cdot A[u^{i}]\,dx & =-\int_{\Omega}\phi_{k}\psi_{i}\,\Delta_{p}u^{i}\,dx\\
 & \qquad\underset{k\to\infty}{\longrightarrow}-\int_{\Omega}\psi_{i}\Delta_{p}u^{i}\,dx,
\end{align*}
which together give the desired result.
\end{proof}
We can similarly show that 
\[
\int_{\Omega}\Delta_{p}u^{i}\,dx=\int_{\partial \Omega}A_{\nu}u^{i}\,d\sigma.
\]
In addition, note that $\int_{\Omega}u^{i}\Delta_{p}u^{i}\,dx$ is meaningful
(since $u^{i}-\varphi^{i}\in W_{0}^{1,p}$ while $\Delta_{p}u^{i}\in W^{-1,p/(p-1)}$
and $\varphi^{i}$ is continuous) and we can similarly show that 
\begin{equation}
\int_{\Omega}u^{i}\Delta_{p}u^{i}+|\nabla u^{i}|^{p}\,dx=\int_{\partial \Omega}\varphi^{i}A_{\nu}u^{i}\,d\sigma.\label{eq: by parts 2}
\end{equation}
Note that $u^{i}=\varphi^{i}$ on $\partial \Omega$.

\begin{lem}
\label{lem: L2 bd der}For $\mathbf{u}\in V$ we have 
\[
\sum_{i=1}^{m}\int_{\Omega}|\nabla u^{i}|^{p}\,dx\le C+C\int_{\partial \Omega}\sum_{i\le m}\psi_{i}(x)A_{\nu}u^{i}\,d\sigma.
\]
\end{lem}

\begin{rem*}
As we will see, the above inequality actually holds for each summand.
Furthermore, with a slight modification of the last part of the proof
we obtain that 
\[
\int_{\Omega}\Delta_{p}u^{i}\,dx\le C+C\int_{\partial \Omega}\psi_{i}(x)A_{\nu}u^{i}\,d\sigma.
\]
\end{rem*}
\begin{proof}
Let $\mathbf{h}_{0}$ be the vector-valued function in $\Omega$ satisfying
$\Delta_{p}h_{0}^{i}=0$, and taking the boundary values $\boldsymbol{\varphi}$
on $\partial \Omega$. Note that $h_{0}^{i}$ is $C^{1}$ and we can plug
it in \eqref{eq: by parts}. By subtracting the resulting relation
from \eqref{eq: by parts 2} we get 
\[
\int_{\Omega}(u^{i}-h_{0}^{i})\Delta_{p}u^{i}\,dx+\int_{\Omega}\nabla(u^{i}-h_{0}^{i})\cdot A[u^{i}]\,dx=0.
\]
Hence 
\begin{align*}
\int_{\Omega}|\nabla u^{i}|^{p}\,dx & =\int_{\Omega}\nabla u^{i}\cdot A[u^{i}]\,dx=\int_{\Omega}(h_{0}^{i}-u^{i})\Delta_{p}u^{i}\,dx+\int_{\Omega}\nabla h_{0}^{i}\cdot A[u^{i}]\,dx\\
 & \le\int_{\Omega}h_{0}^{i}\Delta_{p}u^{i}\,dx+C\int_{\Omega}|\nabla h_{0}^{i}|^{p}\,dx+\frac{1}{2}\int_{\Omega}|A[u^{i}]|^{\frac{p}{p-1}}\,dx\\
 & \le C\int_{\Omega}\Delta_{p}u^{i}\,dx+C+\frac{1}{2}\int_{\Omega}|\nabla u^{i}|^{p}\,dx,
\end{align*}
where we have used the facts that $u^{i},\Delta_{p}u^{i}\ge0$ and
$|A[u^{i}]|^{\frac{p}{p-1}}=|\nabla u^{i}|^{p}$. Thus we have 
\[
\int_{\Omega}|\nabla u^{i}|^{p}\,dx\le C\int_{\Omega}\Delta_{p}u^{i}\,dx.
\]
But since $\psi_{i},\Delta_{p}u^{i}\ge0$ we get 
\[
\int_{\Omega}|\nabla u^{i}|^{p}\,dx\le C\int_{\Omega}\Delta_{p}u^{i}\,dx\le CC_{i}\int_{\Omega}\psi_{i}\,\Delta_{p}u^{i}\,dx,
\]
where $C_{i}=\max_{\overline{\Omega}}\frac{1}{\psi_{i}}>0$. Hence by \eqref{eq: by parts}
we get 
\begin{align*}
\int_{\Omega}|\nabla u^{i}|^{p}\,dx & \le C\int_{\Omega}\psi_{i}\,\Delta_{p}u^{i}\,dx\\
 & =-C\int_{\Omega}\nabla\psi_{i}\cdot A[u^{i}]\,dx+C\int_{\partial \Omega}\psi_{i}A_{\nu}u^{i}\,d\sigma\\
 & \le\tilde{C}\int_{\Omega}|\nabla\psi_{i}|^{p}\,dx+\frac{1}{2}\int_{\Omega}|A[u^{i}]|^{\frac{p}{p-1}}\,dx+C\int_{\partial \Omega}\psi_{i}A_{\nu}u^{i}\,d\sigma\\
 & \le C+\frac{1}{2}\int_{\Omega}|\nabla u^{i}|^{p}\,dx+C\int_{\partial \Omega}\psi_{i}A_{\nu}u^{i}\,d\sigma,
\end{align*}
which gives the desired. 
\end{proof}
\begin{thm}
There exists a minimizer $\mathbf{u}_{\varepsilon}^{\delta}\in V_{\delta}$
for $J_{\varepsilon}$. 
\end{thm}

\begin{proof}
Let $\{\mathbf{u}_{k}\}\subset V_{\delta}$ be a minimizing sequence.
Then by the above lemma and \eqref{eq: Gamma coercive} we have 
\begin{align*}
\sum_{i=1}^{m}\int_{\Omega}|\nabla u_{k}^{i}|^{p}\,dx & \le C+C\int_{\partial \Omega}\sum_{i\le m}\psi_{i}A_{\nu}u_{k}^{i}\,d\sigma\\
 & \le C+C\int_{\partial \Omega}\Gamma\big(x,A_{\nu}u_{k}^{1},\dots,A_{\nu}u_{k}^{m}\big)+C\,d\sigma\\
 & \le C+CJ_{\varepsilon}(\mathbf{u}_{k}).
\end{align*}
Hence $\|\nabla\mathbf{u}_{k}\|_{L^{p}}$ is bounded. In addition,
for the dual exponent $q=\frac{p}{p-1}$ we can see that $\|A[\mathbf{u}_{k}]\|_{L^{q}}=\|\nabla\mathbf{u}_{k}\|_{L^{p}}^{p-1}$
is also bounded. Hence, up to a subsequence, we can assume that $\nabla u_{k}^{i}\rightharpoonup\nabla u^{i}$
in $L^{p}$, $A[u_{k}^{i}]\rightharpoonup A[u^{i}]$ in $L^{q}$,
and $\mathbf{u}_{k}\to\mathbf{u}$ a.e in $\Omega$. Thus $u^{i}\ge0$.
Also, $u^{i}=\varphi^{i}$ on $\partial \Omega$, since $u_{k}^{i}-\varphi^{i}\in W_{0}^{1,p}(\Omega)$,
which is a closed and convex set, hence weakly closed.%
{} Finally, to see that $\Delta_{p}u^{i}$ has the desired properties,
notice that for an appropriate test function $\phi$ we have 
\[
\int_{\Omega}\nabla\phi\cdot A[u^{i}]\,dx=\lim_{k\to\infty}\int_{\Omega}\nabla\phi\cdot A[u_{k}^{i}]\,dx
\]
due to the weak convergence of $A[\mathbf{u}_{k}]$. Therefore $\mathbf{u}\in V_{\delta}$.
Now we can repeat the proof of Lemma 3.3 in \cite{teixeira2005nonlinear}
to deduce the weak lower semicontinuity of $J_{\varepsilon}$ with
respect to this sequence, and conclude the proof (the convexity of $\Gamma$ is needed here). 
\end{proof}

\begin{lem}[\textbf{Hopf's lemma for $p$-harmonic functions}]
\label{lem: Hopf}Suppose $h$ is a $p$-harmonic function on $B_{1}(0)$ with nonnegative
boundary values on $\partial B_{1}$. Then we have 
\[
h(x)\ge c(n,p)\,\mathrm{dist}(x,\partial B_{1})\sup_{B_{1/2}}h.
\]
\end{lem}
\begin{proof}
Consider the function $g(x)=e^{-\lambda|x|^{2}}-e^{-\lambda}$ for some $\lambda>0$. Note
that $g=0$ on $\partial B_{1}$, and $0<g<1$ on $B_{1}$. We also
have 
\[
\partial_{i}g=-2\lambda x_{i}e^{-\lambda|x|^{2}},\qquad \partial_{ij}g=(4\lambda^{2}x_{i}x_{j}-2\lambda\delta_{ij})e^{-\lambda|x|^{2}}.
\]
Now we have $\Delta g=(4\lambda^{2}|x|^{2}-2n\lambda)e^{-\lambda|x|^{2}}$,
and 
\begin{align*}
\Delta_{\infty}g:=\sum_{i,j}\partial_{i}g\partial_{j}g\partial_{ij}g & =\sum_{i,j}4\lambda^{2}(4\lambda^{2}x_{i}^{2}x_{j}^{2}-2\lambda\delta_{ij}x_{i}x_{j})e^{-3\lambda|x|^{2}}\\
 & =4\lambda^{2}(4\lambda^{2}|x|^{4}-2\lambda|x|^{2})e^{-3\lambda|x|^{2}}.
\end{align*}
Therefore 
\begin{align*}
\Delta_{p}g & =\mathrm{div}(|\nabla g|^{p-2}\nabla g)=|\nabla g|^{p-4}\big(|\nabla g|^{2}\Delta g+(p-2)\Delta_{\infty}g\big)\\
 & =(2\lambda|x|)^{p-4}\big(4\lambda^{2}|x|^{2}(4\lambda^{2}|x|^{2}-2n\lambda)\\
 & \qquad\qquad\qquad\qquad+(p-2)4\lambda^{2}(4\lambda^{2}|x|^{4}-2\lambda|x|^{2})\big)e^{-(p-1)\lambda|x|^{2}}\\
 & =(2\lambda)^{p-1}|x|^{p-2}\big(2\lambda|x|^{2}-n+(p-2)(2\lambda|x|^{2}-1)\big)e^{-(p-1)\lambda|x|^{2}}\\
 & =(2\lambda)^{p-1}|x|^{p-2}\big(2(p-1)\lambda|x|^{2}-n-p+2\big)e^{-(p-1)\lambda|x|^{2}}.
\end{align*}
Thus for $\frac{1}{2}\le|x|\le1$ and large enough $\lambda$ we have
\[
\Delta_{p}g\ge2\lambda^{p-1}\big((p-1)\lambda/2-n-p+2\big)e^{-(p-1)\lambda}>0.
\]

Now we have $h\ge\inf_{\overline{B}_{1/2}}h>(\inf_{\overline{B}_{1/2}}h)g$
on $\overline{B}_{1/2}$ (note that $h$ is positive on $B_{1}$ by
maximum principle), and on $B_{1}-\overline{B}_{1/2}$ we have $\Delta_{p}h=0<\Delta_{p}g$.
Also on $\partial B_{1}$ we have $h\ge0=g$. Hence by the maximum
principle we have 
$h(x)\ge g(x)(\inf_{\overline{B}_{1/2}}h)$
for $x\in B_{1}$. But by the Harnack's inequality we have 
\[
\inf_{\overline{B}_{1/2}}h\ge C\sup_{\overline{B}_{1/2}}h
\]
for some constant $C$ which does not depend on $h$. Hence we obtain
\[
h(x)\ge Cg(x)\sup_{\overline{B}_{1/2}}h.
\]
On the other hand note that 
\begin{align*}
g(x) & =g(x)-g(x/|x|)=\int_{1/|x|}^{1}\frac{d}{dt}g(tx)\,dt=\int_{1/|x|}^{1}x\cdot\nabla g(tx)\,dt\\
 & =\int_{1/|x|}^{1}-2\lambda t|x|^{2}e^{-\lambda t^{2}|x|^{2}}\,dt=2\lambda|x|^{2}\int_{1}^{1/|x|}te^{-\lambda t^{2}|x|^{2}}\,dt\\
 & \ge2\lambda|x|^{2}\int_{1}^{1/|x|}te^{-\lambda}\,dt=\lambda e^{-\lambda}|x|^{2}\Big(\frac{1}{|x|^{2}}-1\Big)=\lambda e^{-\lambda}(1-|x|^{2})\\
 & \ge\lambda e^{-\lambda}(1-|x|)=\lambda e^{-\lambda}\,\mathrm{dist}(x,\partial B_{1}),
\end{align*}
which gives the desired. 
\end{proof}
If $h$ is a $p$-harmonic function on $B_{r}(x_{0})$, then $\tilde{h}(x):=h(x_{0}+rx)$
is a $p$-harmonic function on $B_{1}(0)$. Hence we have 
\begin{align*}
h(x_{0}+rx)&=\tilde{h}(x)  \ge c(n,p)\,\mathrm{dist}(x,\partial B_{1})\sup_{B_{1/2}}\tilde{h}\\
 & =c(n,p)\,(1-|x|)\sup_{B_{1/2}}\tilde{h} =c(n,p)\,\frac{(r-r|x|)}{r}\sup_{B_{r/2}(x_{0})}h\\
 & =c(n,p)\,\mathrm{dist}\big(x_{0}+rx,\partial B_{r}(x_{0})\big)\frac{1}{r}\sup_{B_{r/2}(x_{0})}h.
\end{align*}

\begin{lem}
\label{lem: |u is 0| bdd}Let $w\in W^{1,p}(\Omega)$ be a nonnegative %
function. Then there exists $c>0$, depending only on $p$ and the
dimension, such that for any ball $\overline{B}_{r}(x_{0})\subset \Omega$
we have 
\[
\Big(\frac{1}{r}\sup_{B_{r/2}(x_{0})}h\Big)^{p}\cdot\big|B_{r}(x_{0})\cap\{w=0\}\big|\le c\int_{B_{r}(x_{0})}|\nabla(w-h)|^{p}\,dy,
\]
where $h$ satisfies $\Delta_{p}h=0$ in $B_{r}(x_{0})$ taking boundary
values equal to $w$ on $\partial B_{r}(x_{0})$. 
\end{lem}

\begin{proof}
Let $\tau\in(0,1)$ be fixed. For $\xi$ with $|\xi|=1$ we set 
\[
t_{\xi}:=\inf\{t\in[\tau r,r]:w(x_{0}+t\xi)=0\}
\]
provided that this set is nonempty. Otherwise we set $t_{\xi}:=r$.
Now note that $w-h$ and $w$ are absolutely continuous in almost
every direction $\xi$; in particular we have $w(x_{0}+t_{\xi}\xi)=0$
(note that this will not be necessary true if we allow $\tau$ to
be zero). Also $w-h$ is $\mathcal{H}^{n-1}$-a.e. zero on $\partial B_{r}(x_{0})$
as its trace is zero there, so $(w-h)(x_{0}+r\xi)=0$. Thus for almost
every $\xi$ for which $t_{\xi}<r$ we have
\begin{align*}
h(x_{0}+t_{\xi}\xi) & =(w-h)(x_{0}+r\xi)-(w-h)(x_{0}+t_{\xi}\xi)\\
 & =\int_{t_{\xi}}^{r}\frac{d}{dt}\big((w-h)(x_{0}+t\xi)\big)\,dt\\
 & =\int_{t_{\xi}}^{r}\nabla_{\xi}(w-h)(x_{0}+t\xi)\,dt\\
 & \le(r-t_{\xi})^{\frac{p-1}{p}}\Big(\int_{t_{\xi}}^{r}\big|\nabla(w-h)(x_{0}+t\xi)\big|^{p}\,dt\Big)^{\frac{1}{p}}.
\end{align*}
On the other hand, using Hopf's lemma we get 
\begin{align*}
    h(x_{0}+t_{\xi}\xi) & \ge c(n,p)\,\mathrm{dist}\big(x_{0}+t_{\xi}\xi,\,\partial B_{r}(x_{0})\big)\frac{1}{r}\sup_{B_{r/2}(x_{0})}h \\ & =c(n,p)(r-t_{\xi})\frac{1}{r}\sup_{B_{r/2}(x_{0})}h.
\end{align*}

Hence we obtain 
\[
(r-t_{\xi})\Big(\frac{1}{r}\sup_{B_{r/2}(x_{0})}h\Big)^{p}\le C(n,p)\int_{t_{\xi}}^{r}\big|\nabla(w-h)(x_{0}+t\xi)\big|^{p}\,dt.
\]
Note that this inequality is trivially satisfied if $t_{\xi}=r$.
Now by integrating over $\xi$ we get 
\begin{align*}
 & \hspace{-1cm}C(n,p)\int_{B_{r}(x_{0})}\big|\nabla(w-h)(x)\big|^{p}\,dx\\
 & \ge C(n,p)\int_{\partial B_{1}(0)}\int_{t_{\xi}}^{r}\big|\nabla(w-h)(x_{0}+t\xi)\big|^{p}\,dtd\xi\\
 & \ge\Big(\frac{1}{r}\sup_{B_{r/2}(x_{0})}h\Big)^{p}\int_{\partial B_{1}(0)}(r-t_{\xi})\,d\xi \allowdisplaybreaks \\
 & =\Big(\frac{1}{r}\sup_{B_{r/2}(x_{0})}h\Big)^{p}\int_{\partial B_{1}(0)}\int_{t_{\xi}}^{r}1\,dtd\xi\\
 & \ge\Big(\frac{1}{r}\sup_{B_{r/2}(x_{0})}h\Big)^{p}\int_{B_{r}(x_{0})-B_{\tau r}(x_{0})}\chi_{\{w=0\}}\,dx,
\end{align*}
where the last inequality follows from the definition of $t_{\xi}$.
Finally, we get the desired by letting $\tau\to0$. 
\end{proof}

\begin{lem}
\label{lem: aux obst pr} Let $\mathbf{u}=\mathbf{u}_{\varepsilon}^{\delta}$
be a minimizer of $J_{\varepsilon}$ over $V_{\delta}$, and $B\subset \Omega$
be a ball. Then there exists a unique $v^{i}\in W^{1,p}(\Omega)$ that
minimizes the functional 
\[
\int_{\Omega}|\nabla v^{i}|^{p}\,dx
\]
among all functions with $v^{i}=\varphi^{i}$ on $\partial \Omega$ and
$v^{i}\le0$ on $\{u^{i}=0\}-B$. The functions $v^{i}$ also satisfy 
\begin{enumerate}
\item $v^{i}=0$ on $\{u^{i}=0\}-B$, 
\item $\mathbf{v}=(v^{1},\dots,v^{m})\in V_{\delta}$, 
\item $0\le u^{i}\le v^{i}\le C_{0}=\max_{\partial \Omega}|\boldsymbol{\varphi}|$, 
\item $\int_{\Omega}v^{i}\,\Delta_{p}v^{i}\,dx=0$. 
\end{enumerate}
\end{lem}

\begin{rem*}
Instead of a ball $B$, we can also use other open subsets of $\Omega$
in the above lemma.
\end{rem*}
\begin{proof}
The proof is similar to that of Lemma 3.7 in \cite{teixeira2005nonlinear} (see also the proof of Theorem 3.6 in \cite{aguilera1987optimization}). 
\end{proof}

\begin{thm}
\label{thm: Du<C}Let $\mathbf{u}=\mathbf{u}_{\varepsilon}^{\delta}$
be a minimizer of $J_{\varepsilon}$ over $V_{\delta}$. There exists
a constant $M=M_{\varepsilon}$, independent of $\delta$, such that
if for some $j$ we have 
\[
\frac{1}{r}\sup_{B_{r/2}(x)}u^{j}\ge M,
\]
then $B_{r}(x)\subset\{|\mathbf{u}|>0\}$, and $\Delta_{p}u^{i}=0$
in $B_{r}(x)$ for every $i$. 
\end{thm}

\begin{proof}
Let $\mathbf{v}\in V_{\delta}$ be the function given by Lemma \ref{lem: aux obst pr}
for $B_{r}(x)$. Then we have 
\[
J_{\varepsilon}(\mathbf{u})\le J_{\varepsilon}(\mathbf{v}).
\]
Let $\mathbf{h}_{0}$ be the vector-valued function in $\Omega$ satisfying
$\Delta_{p}h_{0}^{i}=0$, and taking the boundary values $\boldsymbol{\varphi}$
on $\partial \Omega$. Since $0\le u^{i}\le v^{i}\le h_{0}^{i}$, for each
$z\in\partial \Omega$ we have 
\[
\partial_{\nu}h_{0}^{i}(z)\le\partial_{\nu}v^{i}(z)\le\partial_{\nu}u^{i}(z)\le0.
\]
Then by using the fact that $\mathbf{u},\mathbf{v},\mathbf{h}_{0}$
take the same boundary values and therefore have equal tangential
derivatives on $\partial \Omega$, we  deduce that 
\[
a\le A_{\nu}h_{0}^{i}(z)\le A_{\nu}v^{i}(z)\le A_{\nu}u^{i}(z),
\]
where $a$ is a lower bound for $A_{\nu}h_{0}^{i}$ (note that $a$
does not depend on $\delta$).

Hence by the property (2) of $\Gamma$ we have 
\begin{align}\label{eq: Gamma > sum}
 & \int_{\partial \Omega}\Gamma\big(x,A_{\nu}\mathbf{u}(x)\big)-\Gamma\big(x,A_{\nu}\mathbf{v}(x)\big)\,d\sigma\nonumber \\
 & \qquad=\sum_{i=1}^{m}\int_{\partial \Omega}\Gamma(x,A_{\nu}u^{1},\dots,A_{\nu}u^{i-1},A_{\nu}u^{i},A_{\nu}v^{i+1},\dots,A_{\nu}v^{m})\nonumber \\
 & \qquad\qquad\qquad-\Gamma(x,A_{\nu}u^{1},\dots,A_{\nu}u^{i-1},A_{\nu}v^{i},A_{\nu}v^{i+1},\dots,A_{\nu}v^{m})\,d\sigma \\
 & \qquad\ge C_{a}\sum_{i=1}^{m}\int_{\partial \Omega}A_{\nu}u^{i}-A_{\nu}v^{i}\,d\sigma,\nonumber 
\end{align}
where $C_{a}>0$ is the lower bound of $\partial_{\xi_{i}}\Gamma$'s
on the set $\{(x,\xi):\xi_{i}\ge a\}$. On the other hand, using the
identity (\ref{eq: by parts 2}) we get 
\begin{align}\label{eq: BC> gradient}
C_{0}\int_{\partial \Omega}A_{\nu}u^{i}-A_{\nu}v^{i}\,d\sigma & \ge\int_{\partial \Omega}\varphi^{i}\big(A_{\nu}u^{i}-A_{\nu}v^{i}\big)\,d\sigma \notag\\
 & =\int_{\Omega}u^{i}\,\Delta_{p}u^{i}+|\nabla u^{i}|^{p}\,dy\notag\\
 & \qquad\qquad-\int_{\Omega}v^{i}\,\Delta_{p}v^{i}+|\nabla v^{i}|^{p}\,dy\\
 & \ge\int_{\Omega}|\nabla u^{i}|^{p}\,dy-\int_{\Omega}|\nabla v^{i}|^{p}\,dy,\notag
\end{align}
where $C_{0}=\max_{\partial \Omega}|\boldsymbol{\varphi}|$, and in the
last line we used the facts that $\int_{\Omega}v^{i}\,\Delta_{p}v^{i}\,dy=0$
and $u^{i},\Delta_{p}u^{i}\ge0$. Now consider the function $h^{i}$
in $B_{r}(x)$ satisfying $\Delta_{p}h^{i}=0$, and taking boundary
values equal to $u^{i}$. We extend $h^{i}$ to be equal to $u^{i}$
outside of $B_{r}(x)$. Then we have $\mathbf{h}=(h_{1},\dots,h_{m})\in V_{\delta}$.
In addition, $h^{i}=u^{i}=\varphi^{i}$ on $\partial \Omega$ and $h^{i}=u^{i}=0$
on $\{u^{i}=0\} - B_{r}(x)$. Hence due to the minimality property of
$v^{i}$ given by Lemma \ref{lem: aux obst pr} we have 
\[
\int_{\Omega}|\nabla v^{i}|^{p}\,dy\le\int_{\Omega}|\nabla h^{i}|^{p}\,dy.
\]
Combining this with the above inequality we get 
\begin{align*}
C_{0}\int_{\partial \Omega}A_{\nu}u^{i}-A_{\nu}v^{i}\,d\sigma & \ge\int_{\Omega}|\nabla u^{i}|^{p}-|\nabla v^{i}|^{p}\,dy\\
 & \ge\int_{\Omega}|\nabla u^{i}|^{p}-|\nabla h^{i}|^{p}\,dy\\
 & \ge C\int_{B_{r}(x)}|\nabla(u^{i}-h^{i})|^{p}\,dy,
\end{align*}
where the last inequality can be proved similarly to the proof of Lemma
3.1 in \cite{fotouhi2023minimization}. (Note that in
the last line we have also used the fact that $u^{i}=h^{i}$ outside
$B_{r}(x)$.)

Summing the above inequality for each $i$, and using the facts that
$J_{\varepsilon}(\mathbf{u})\le J_{\varepsilon}(\mathbf{v})$, and
$f_{\varepsilon}$ has Lipschitz constant equal to $\frac{1}{\varepsilon}$
we get 
\begin{align*}
\frac{C_{a}}{C_{0}}\sum_{i\le m}\int_{B_{r}(x)}|\nabla(u^{i}-h^{i})|^{p}\,dy & \le C_{a}\int_{\partial \Omega}\sum_{i\le m}(A_{\nu}u^{i}-A_{\nu}v^{i})\,d\sigma\\
 & \le\int_{\partial \Omega}\Gamma\big(x,A_{\nu}\mathbf{u}(x)\big)-\Gamma\big(x,A_{\nu}\mathbf{v}(x)\big)\,d\sigma\\
 & \le f_{\varepsilon}\big(\big|\{|\mathbf{v}|>0\}\big|\big)-f_{\varepsilon}\big(\big|\{|\mathbf{u}|>0\}\big|\big)\\
 & \le\frac{1}{\varepsilon}\big|B_{r}(x)\cap\{|\mathbf{u}|=0\}\big|,
\end{align*}
since $0\le u^{i}\le v^{i}$, and outside of $B_{r}(x)$, $|\mathbf{u}|=0$
implies $|\mathbf{v}|=0$. Therefore by Lemma \ref{lem: |u is 0| bdd} applied
to $u^{j}$ we obtain 
\begin{align*}
\big|B_{r}(x)\cap\{|\mathbf{u}|=0\}\big| & \ge\frac{\varepsilon C_{a}}{C_{0}}\sum_{i\le m}\int_{B_{r}(x)}|\nabla(u^{i}-h^{i})|^{p}\,dy\\
 & \ge\frac{\varepsilon C_{a}}{C_{0}}\int_{B_{r}(x)}|\nabla(u^{j}-h^{j})|^{p}\,dy\\
 & \ge\frac{\varepsilon C_{a}}{cC_{0}}\Big(\frac{1}{r}\sup_{B_{r/2}(x)}h^{j}\Big)^{p}\cdot\big|B_{r}(x)\cap\{u^{j}=0\}\big|\\
 & \ge\frac{\varepsilon C_{a}}{cC_{0}}\Big(\frac{1}{r}\sup_{B_{r/2}(x)}u^{j}\Big)^{p}\cdot\big|B_{r}(x)\cap\{u^{j}=0\}\big|\\
 & \ge\frac{\varepsilon C_{a}M^{p}}{cC_{0}}\big|B_{r}(x)\cap\{|\mathbf{u}|=0\}\big|,
\end{align*}
since $|\mathbf{u}|=0$ implies $u^{j}=0$, and $h^{j}\ge u^{j}$
as $u^{j}$ is $p$-subharmonic. Hence if $M>(\frac{cC_{0}}{\varepsilon C_{a}})^{\frac{1}{p}}$
then $\big|B_{r}(x)\cap\{|\mathbf{u}|=0\}\big|$ must be zero, as
desired. Note that in this case the above inequality also implies
that $u^{i}=h^{i}$ in $B_{r}(x)$ for each $i$; so $u^{i}$ satisfies
the equation in $B_{r}(x)$. 
\end{proof}
\begin{cor}\label{cor: u_e,d Lip}
All minimizers $\mathbf{u}_{\varepsilon}^{\delta}$
are Lipschitz, and for every $\Omega'\subset\subset \Omega$ there exists
a constant $K_{\varepsilon}=K_{\varepsilon}(\Omega')$, independent
of $\delta$, such that 
\[
\|\mathbf{u}_{\varepsilon}^{\delta}|_{\Omega'}\|_{\mathrm{Lip}}\le K_{\varepsilon}.
\]
In addition, $\Delta_{p}(u_{\varepsilon}^{\delta})^{i}=0$ in the
open set $\{|\mathbf{u}_{\varepsilon}^{\delta}|>0\}$. 
\end{cor}

\begin{proof}
For simplicity we set $\mathbf{u}=\mathbf{u}_{\varepsilon}^{\delta}$.
First let us show that $\{|\mathbf{u}|>0\}$ is an open set. Suppose
$x\in\{|\mathbf{u}|>0\}$. Then $u^{j}(x)>0$ for some $j$. Then
for small enough $r$ we must have 
\[
\frac{1}{r}\sup_{B_{r/2}(x)}u^{j}\ge\frac{1}{r}u^{j}(x)\ge M.
\]
Hence the previous theorem implies that $B_{r}(x)\subset\{|\mathbf{u}|>0\}$
and we have $\Delta_{p}u^{i}=0$ in $B_{r}(x)$.

Next  note that $\nabla\mathbf{u}=0$ a.e. in $\{|\mathbf{u}|=0\}$. So suppose $x\in\{|\mathbf{u}|>0\}\cap\Omega'$. Let $\Omega'\subset\subset\tilde{\Omega}\subset\subset \Omega$,
and $B=B_{d}(x)$, where $d=\mathrm{dist}\big(x,\partial(\{|\mathbf{u}|>0\}\cap\tilde{\Omega})\big)$.
If $\partial B$ touches $\partial\{|\mathbf{u}|=0\}$ then $B_{d+d'}(x)$
intersects $\{|\mathbf{u}|=0\}$, and by previous theorem we have
\[
\frac{1}{d+d'}\sup_{B_{(d+d')/2}(x)}u^{i}\le M
\]
for every $i$. Hence in the limit $d'\to0$ we get 
\[
\frac{1}{d}\sup_{B_{d/2}(x)}u^{i}\le M.
\]
Now since $u^{i}$'s are $p$-harmonic in $B$, as shown in the proof
of Lemma 3.1 of \cite{danielli2005minimum}, we have 
\[
|\nabla u^{i}(x)|\le C\frac{1}{d}\sup_{B_{d/2}(x)}u^{i}\le CM,
\]
where the constant $C$ depends only on $p$ and the dimension $n$.
On the other hand, if $\partial B$ touches $\partial\tilde{\Omega}$ then,
by the interior derivative estimate of \cite{lewis1983regularity},
we obtain (the dependence on $d$ follows from the proof of this estimate; see equation (3.4) in \cite{lewis1983regularity})
\[
|\nabla u^{i}(x)|\le \frac{C(n,p)}{d^n}\|\mathbf{u}\|_{W^{1,p}}\le C,
\]
since $d\ge\mathrm{dist}(\Omega',\partial\tilde{\Omega})$, and $\|\mathbf{u}\|_{W^{1,p}}$ is bounded independently of $\delta$
as will be shown now. Let $\Omega'\subset\subset \Omega$ be a smooth open set
with $|\Omega-\Omega'|=1$. Let $\mathbf{u}_{0}$ be a vector-valued function
on $\Omega-\Omega'$ that satisfies the equation $\Delta_{p}u_{0}^{i}=0$, and
takes the boundary values $\boldsymbol{\varphi}$ on $\partial \Omega$
and $0$ on $\partial \Omega'$. Then for every small enough $\delta$
we have $\mathbf{u}_{0}\in V_{\delta}$. Hence (remember that $\mathbf{u}=\mathbf{u}_{\varepsilon}^{\delta}$)
\begin{align*}
C=J_{\varepsilon}(\mathbf{u}_{0})\ge J_{\varepsilon}(\mathbf{u}_{\varepsilon}^{\delta}) & \ge\int_{\partial \Omega}\Gamma\big(x,A_{\nu}\mathbf{u}_{\varepsilon}^{\delta}(x)\big)\,d\sigma\\
 & \ge\int_{\partial \Omega}\sum_{i=1}^{m}\psi_{i}(x)A_{\nu}(u_{\varepsilon}^{\delta})^{i}-C\,d\sigma,
\end{align*}
where we used \eqref{eq: Gamma coercive} in the last line. Thus by
Lemma \ref{lem: L2 bd der} the $\|\nabla\mathbf{u}_{\varepsilon}^{\delta}\|_{L^{p}(\Omega;\mathbb{R}^{m})}$
is bounded as $\delta\to0$, and the boundedness of $\|\mathbf{u}_{\varepsilon}^{\delta}\|_{W^{1,p}}$ follows from Poincare inequality and the fact that all of $\mathbf{u}_{\varepsilon}^{\delta}$'s have the same boundary values.

Finally, to see that $\mathbf{u}$ is Lipschitz continuous on all
of $\Omega$, note that $\mathbf{u}$ has $p$-harmonic components near
the smooth boundary $\partial \Omega$, attaining smooth boundary conditions
$\boldsymbol{\varphi}$; hence the gradient of $\mathbf{u}$ is bounded
near the boundary too. 
\end{proof}

\begin{lem}
\label{lem: u>0 near bdry} There exists $\delta_{0}=\delta_{0}(\varepsilon)>0$, such that for
every $\delta$ we have $|\mathbf{u}_{\varepsilon}^{\delta}|>0$
on $\Omega_{\delta_{0}}$.
\end{lem}

\begin{rem*}
Note that as a consequence, $\Delta_{p}(u_{\varepsilon}^{\delta})^{i}=0$
on $\Omega_{\delta_{0}}$ for every $\delta$ (by Theorem \ref{thm: Du<C}).
In other words, $\mathbf{u}_{\varepsilon}^{\delta}\in V_{\delta_{0}}$
for every $\delta$.
\end{rem*}

\begin{proof}
Suppose to the contrary that there is a sequence $\mathbf{u}_{k}=\mathbf{u}_{\varepsilon}^{\delta_{k}}$
for which we have 
\[
2d_{k}:=\mathrm{dist}(\{|\mathbf{u}_{k}|=0\},\partial \Omega)\to0.
\]
Then the midpoint of the closest points on $\{|\mathbf{u}_{k}|=0\}$
and $\partial \Omega$, which we call $x_{k}$, has distance $d_{k}$ from
both of these sets. So the boundary of the ball $B_{d_{k}}(x_{k})$
touches both of these sets. In addition, by Theorem \ref{thm: Du<C},
for every $t>0$ we must have 
\[
\frac{1}{d_{k}}\sup_{B_{d_{k}/2}(x_{k}+t\nu_{k})}u_{k}^{i}\le M_{\varepsilon}
\]
for every $i$ (here $\nu_{k}$ is the direction of the line segment
from $x_{k}$ to its closest point on $\{|\mathbf{u}_{k}|=0\}$).
So in the limit $t\to0$ we get 
\begin{equation}
\sup_{B_{d_{k}/2}(x_{k})}u_{k}^{i}\le M_{\varepsilon}d_{k}.\label{eq: sup u_k -> 0}
\end{equation}
We also have \begin{equation*}
\sup_{B_{d_{k}}(x_{k})}|\mathbf{u}_{k}|\ge c_{0},    
\end{equation*}
where $c_{0}=\min_{i}\min_{\partial \Omega}\varphi^{i}>0$. Because at the point $y_{k}\in\partial B_{d_{k}}(x_{k})\cap\partial \Omega$ we have $u_{k}^{i}(y_{k})=\varphi^{i}(y_{k})\ge c_{0}$ (note that $u_{k}^{i}$ is continuous up to the boundary).

Next consider the functions 
\[
\hat{\mathbf{u}}_{k}(x):=\frac{\mathbf{u}(x_{k}+d_{k}x)}{\sup_{B_{d_{k}}(x_{k})}|\mathbf{u}_{k}|}
\]
on $B_{1}$. Then $\hat{u}_{k}^{i}$ is positive and $p$-harmonic
on $B_{1}$, and we have $\sup_{B_{1}}|\hat{\mathbf{u}}_{k}|=1$.
In addition,
by (\ref{eq: sup u_k -> 0}) we have 
\[
\sup_{B_{1/2}}\hat{u}_{k}^{i}=\frac{\sup_{B_{d_{k}/2}(x_{k})}u_{k}^{i}}{\sup_{B_{d_{k}}(x_{k})}|\mathbf{u}_{k}|}\le\frac{M_{\varepsilon}d_{k}}{c_{0}}\underset{k\to\infty}{\longrightarrow}0.
\]
Furthermore, note that $\hat{u}_{k}^{i}$ is a uniformly bounded sequence
of $p$-harmonic functions on $B_{1}$, so there is $\alpha>0$ such
that for all $r<1$ the H\"older norms $\|\hat{u}_{k}^{i}\|_{C^{0,\alpha}(\overline{B}_{r})}$
are uniformly bounded (see page 251 of \cite{MR0244627}). Hence,
by a diagonal argument, we can construct a subsequence of $\hat{u}_{k}^{i}$,
which we still denote by $\hat{u}_{k}^{i}$, that locally uniformly
converges to a nonnegative $p$-harmonic function $\hat{u}_{\infty}^{i}$
on $B_{1}$. In addition, $\hat{u}_{\infty}^{i}$ must vanish on $B_{1/2}$
by the above estimate. Thus by the strong maximum principle we must
have $\hat{\mathbf{u}}_{\infty}\equiv0$ on $B_{1}$.  

Now for $y_{k}\in\partial B_{d_{k}}(x_{k})\cap\partial \Omega$ and $r<d_{k}$
we have 
\[
\underset{B_{r}(y_{k})\cap \Omega}{\mathrm{osc}}\,u_{k}^{i}\le C(n,p)\Big(r^{\alpha}+\underset{B_{r}(y_{k})\cap\partial \Omega}{\mathrm{osc}}\,\varphi^{i}\Big)\le Cr^{\alpha}
\]
for some $\alpha\in(0,1)$. This estimate holds by Theorem 4.19 of
\cite{maly1997fine} when $1<p\le n$. And when $p>n$ this estimate
holds due to the uniform H\"older continuity of $u_{k}^{i}$ on $\overline{\Omega}$,
since $\|\mathbf{u}_{k}\|_{W^{1,p}(\Omega)}$ is uniformly bounded as we
have seen in the proof of Corollary \ref{cor: u_e,d Lip}. Hence for
$r=d_{k}/2$ we have 
\[
\min_{B_{d_{k}/2}(y_{k})\cap B_{d_{k}}(x_{k})}u_{k}^{i}\ge\min_{B_{d_{k}/2}(y_{k})\cap \Omega}u_{k}^{i}\ge\frac{1}{2}c_{0},
\]
where $c_{0}=\min_{i}\min_{\partial \Omega}\varphi^{i}$. Therefore for
$\hat{y}_{k}=\frac{1}{d_{k}}(y_{k}-x_{k})\in\partial B_{1}$ we have
\[
\min_{B_{1/2}(\hat{y}_{k})\cap B_{1}}\hat{u}_{k}^{i}=\frac{1}{\sup_{B_{d_{k}}(x_{k})}|\mathbf{u}_{k}|}\min_{B_{d_{k}/2}(y_{k})\cap B_{d_{k}}(x_{k})}u_{k}^{i}\ge c>0,
\]
since $\sup_{B_{d_{k}}(x_{k})}|\mathbf{u}_{k}|\le mC_{0}$ where $C_{0}=\max_{\partial \Omega}|\boldsymbol{\varphi}|$.
Thus $\hat{u}_{k}^{i}$ has a uniform positive lower bound on a subset
of $B_{1}$ with positive volume (where the volume is independent
of $k$). So no subsequence of $\hat{\mathbf{u}}_{k}$ can converge
locally uniformly to $\hat{\mathbf{u}}_{\infty}\equiv0$, because
otherwise they will uniformly converge to $0$ outside a set of small
volume, contradicting the uniform boundedness from below.
\end{proof}

Now we can find a minimizer for $J_{\varepsilon}$ over $V$. 

\begin{thm}
\label{thm: Lip-reg} There exists a minimizer $\mathbf{u}_{\varepsilon}\in V$ for $J_{\varepsilon}$.
Moreover, $\mathbf{u}_{\varepsilon}$ is a Lipschitz function, and
$\Delta_{p}u_{\varepsilon}^{i}=0$ in the open set $\{|\mathbf{u}_{\varepsilon}|>0\}$. 
\end{thm}

\begin{rem*}
As we will see in the following proof, $\mathbf{u}_{\varepsilon}^{\delta}\in V_{\delta_{0}}$ for $\delta_{0}=\delta_{0}(\varepsilon)$ given by the above lemma. So in fact $\mathbf{u}_{\varepsilon}$ is a minimizer of $J_{\varepsilon}$ over some $V_{\delta}$, and therefore it has all the properties of $\mathbf{u}_{\varepsilon}^{\delta}$'s that we have proved so far. In particular, we have $|\mathbf{u}_{\varepsilon}|>0$ on $\Omega_{\delta_{0}}$.
\end{rem*}

\begin{proof}
As we have shown in the proof of Corollary \ref{cor: u_e,d Lip},
 $\|\nabla\mathbf{u}_{\varepsilon}^{\delta}\|_{L^{p}(\Omega;\mathbb{R}^{m})}$
is bounded as $\delta\to0$. Hence there is a subsequence such that
$\mathbf{u}_{\varepsilon}^{\delta}\rightharpoonup\mathbf{u}_{\varepsilon}$
weakly in $W^{1,p}$ (and also a.e.) with $A(\nabla(u_{\varepsilon}^{\delta})^{i})\rightharpoonup A(\nabla u_{\varepsilon}^{i})$
in $L^{q}$ as $\delta\to0$. So, in particular, $u_{\varepsilon}^{i}\ge0$,
$u_{\varepsilon}^{i}$ is $p$-subharmonic, and attains the boundary
condition $\varphi^{i}$. Furthermore, by Corollary \ref{cor: u_e,d Lip}, $\mathbf{u}_{\varepsilon}^{\delta}\to\mathbf{u}_{\varepsilon}$
uniformly on compact subsets of $\Omega$. Hence for each ball $\overline{B}\subset\{|\mathbf{u}_{\varepsilon}|>0\}$
and all small enough $\delta$ we have $\overline{B}\subset\{|\mathbf{u}_{\varepsilon}^{\delta}|>0\}$.
Therefore by using test functions with support in $B$ together with
$A(\nabla(u_{\varepsilon}^{\delta})^{i})\rightharpoonup A(\nabla u_{\varepsilon}^{i})$
we can conclude that $u_{\varepsilon}^{i}$ is $p$-harmonic in $B$.

The same reasoning applied to test functions with support in $\Omega_{\delta_{0}}$, for $\delta_{0}$ given by the previous lemma, implies that $u_{\varepsilon}^{i}$
is $p$-harmonic in $\Omega_{\delta_{0}}$, and thus $\mathbf{u}_{\varepsilon}\in V_{\delta_{0}}\subset V$.
In particular, $u_{\varepsilon}^{i}$ is $p$-harmonic near the smooth
boundary $\partial \Omega$, attaining smooth boundary conditions $\varphi^{i}$,
so it is Lipschitz near $\partial \Omega$. Moreover, $\mathbf{u}_{\varepsilon}$
is Lipschitz inside $\Omega$ away from its boundary, because it is the
uniform limit of a sequence of Lipschitz functions with uniformly
bounded Lipschitz constants. Hence $\mathbf{u}_{\varepsilon}$ is
Lipschitz on all of $\Omega$.

Finally note that $\mathbf{u}_{\varepsilon}$ minimizes $J_{\varepsilon}$
over $V$, since for every $\mathbf{w}\in V$ we have $\mathbf{w}\in V_{\delta}$
for some $\delta$. Thus $J_{\varepsilon}(\mathbf{u}_{\varepsilon}^{\delta})\le J_{\varepsilon}(\mathbf{w})$.
However, $\mathbf{u}_{\varepsilon}^{\delta}\to\mathbf{u}_{\varepsilon}$,
so we get $J_{\varepsilon}(\mathbf{u}_{\varepsilon})\le J_{\varepsilon}(\mathbf{w})$
due to the semicontinuity of $J_{\varepsilon}$. 
\end{proof}

\section{Regularity of solutions to the penalized problem}

To simplify the notation, throughout this section we will suppress the index $\varepsilon$ in $\mathbf{u}_{\varepsilon}$.
\begin{thm}
\label{thm: Du>c} 
For $\tau\in(0,1/4)$ there exists $m_{\varepsilon}(\tau)$
such that if for each $i$ we have 
\[
\frac{1}{r}\sup_{B_{r/2}(x)}u^{i}\le m_{\varepsilon}(\tau),
\]
then $B_{\tau r}(x)\subset\{|\mathbf{u}|=0\}$. 
\end{thm}

\begin{proof}
Similarly to Lemma \ref{lem: aux obst pr}, we can show that there
is $v^{i}\in W^{1,p}(\Omega)$ that minimizes the functional $\int_{\Omega}|\nabla v^{i}|^{p}\,dx$
among all functions with $v^{i}=\varphi^{i}$ on $\partial \Omega$ and
$v^{i}\le0$ on $\{u^{i}=0\}\cup\overline{B}_{\tau r}(x)$. The function
$v^{i}$ also satisfies $\Delta_{p}v^{i}\ge0$, $\int_{\Omega}v^{i}\,\Delta_{p}v^{i}\,dx=0$,
and $u^{i}\ge v^{i}\ge0$ (to see this, note that $\Delta_{p}v^{i}\ge\Delta_{p}u^{i}$
on $\Omega-\big(\{u^{i}=0\}\cup\overline{B}_{\tau r}(x)\big)\subset\{|\mathbf{u}|>0\}$,
and $v^{i}-u^{i}\le0$ on $\{u^{i}=0\}\cup\overline{B}_{\tau r}(x)$
or $\partial \Omega$). In addition, we have $\mathbf{v}=(v_{1},\dots,v_{m})\in V_{\delta_{1}}\subset V$
(where $\delta_{1}$ is small enough so that $\overline{B}_{\tau r}(x)\subset \Omega-\Omega_{\delta_{1}}$).
Thus 
\[
J_{\varepsilon}(\mathbf{u})\le J_{\varepsilon}(\mathbf{v}).
\]
Let us assume that $\delta_{1}$ is small enough so that $\overline{B}_{r}(x)\subset \Omega-\Omega_{\delta_{1}}$
and $\mathbf{u}\in V_{\delta_{1}}$. Let $\mathbf{w}$ be a vector-valued
$p$-harmonic function in $\Omega_{\delta_{1}}$ with boundary values
equal to $\boldsymbol{\varphi}$ on $\partial \Omega$ and equal to 0 on
$\partial \Omega_{\delta_{1}}-\partial \Omega$. Then we have $u^{i}\ge v^{i}\ge w^{i}\ge0$
(since $\mathbf{u},\mathbf{v}$ are also $p$-harmonic on $\Omega_{\delta_{1}}$,
and nonnegative everywhere). Thus for each $z\in\partial \Omega$ we have
\[
0\ge\partial_{\nu}w^{i}(z)\ge\partial_{\nu}v^{i}(z)\ge\partial_{\nu}u^{i}(z).
\]
Next using the fact that $\mathbf{u},\mathbf{v},\mathbf{w}$ take
the same boundary values on $\partial \Omega$, and therefore have equal
tangential derivatives on $\partial \Omega$, we  deduce that 
\[
0\ge A_{\nu}w^{i}(z)\ge A_{\nu}v^{i}(z)\ge A_{\nu}u^{i}(z).
\]
Now similar to (\ref{eq: Gamma > sum}) we  can show that 
\[
\int_{\partial \Omega}\Gamma\big(x,A_{\nu}\mathbf{v}(x)\big)-\Gamma\big(x,A_{\nu}\mathbf{u}(x)\big)\,d\sigma\le C_{1}\sum_{i=1}^{m}\int_{\partial \Omega}A_{\nu}v^{i}-A_{\nu}u^{i}\,d\sigma,
\]
where $C_{1}>0$ is the upper bound of $\partial_{\xi_{i}}\Gamma$'s
on the set $\{(x,\xi):\xi_{i}\le0\}$. On the other hand, using the
identity (\ref{eq: by parts 2}) we obtain  (using the notation $c_{0}=\min_{i}\min_{\partial \Omega}\varphi^{i}$)
\begin{align}\label{boundary-to-gradient-v<u}
c_{0}\int_{\partial \Omega}A_{\nu}v^{i}-A_{\nu}u^{i}\,d\sigma & \le\int_{\partial \Omega}\varphi^{i}\big(A_{\nu}v^{i}-A_{\nu}u^{i}\big)\,d\sigma\notag\\
 & =\int_{\Omega}v^{i}\,\Delta_{p}v^{i}+|\nabla v^{i}|^{p}\,dy\notag\\
 & \qquad\qquad-\int_{\Omega}u^{i}\,\Delta_{p}u^{i}+|\nabla u^{i}|^{p}\,dy\notag\\
 & =\int_{\Omega}|\nabla v^{i}|^{p}\,dy-\int_{\Omega}|\nabla u^{i}|^{p}\,dy,
\end{align}
where in the last line we used the facts that $\int_{\Omega}v^{i}\,\Delta_{p}v^{i}\,dy=0$,
and $\Delta_{p}u^{i}=0$ on $\{u^{i}\ne0\}\subset\{|\mathbf{u}|>0\}$.

Summing the above inequality for each $i$, and using the facts that
$J_{\varepsilon}(\mathbf{u})\le J_{\varepsilon}(\mathbf{v})$, and
the derivative of $f_{\varepsilon}$ is bounded below by $\varepsilon$,
we get 
\begin{align}
\frac{C_{1}}{c_{0}}\sum_{i\le m}\int_{\Omega}|\nabla v^{i}|^{p}-|\nabla u^{i}|^{p}\,dy & \ge C_{1}\int_{\partial \Omega}\sum_{i\le m}(A_{\nu}v^{i}-A_{\nu}u^{i})\,d\sigma\nonumber \\
 & \ge\int_{\partial \Omega}\Gamma\big(x,A_{\nu}\mathbf{v}(x)\big)-\Gamma\big(x,A_{\nu}\mathbf{u}(x)\big)\,d\sigma\nonumber \\
 & \ge f_{\varepsilon}\big(\big|\{|\mathbf{u}|>0\}\big|\big)-f_{\varepsilon}\big(\big|\{|\mathbf{v}|>0\}\big|\big)\nonumber \\
 & \ge\varepsilon\big|\{|\mathbf{u}|>0\}\cap\{|\mathbf{v}|=0\}\big|\nonumber \\
 & \ge\varepsilon\big|\{|\mathbf{u}|>0\}\cap B_{\tau r}(x)\big|,\label{eq: 1}
\end{align}
since $u^{i}\ge v^{i}\ge0$ and $v^{i}=0$ in $B_{\tau r}(x)$.

Next we define $g:(0,\infty)\to\mathbb{R}$ by 
\[
g(t):=\begin{cases}
t^{\frac{p-n}{p-1}}-(\tau r)^{\frac{p-n}{p-1}} & p>n,\\
\log t-\log(\tau r) & p=n,\\
(\tau r)^{\frac{p-n}{p-1}}-t^{\frac{p-n}{p-1}} & p<n.
\end{cases}
\]
Note that $g$ is an increasing function that vanishes at $t=\tau r$,
and is negative for $t<\tau r$. In addition, $g(|x|)$ is a $p$-harmonic
function in $\mathbb{R}^{n}-\{0\}$, which is negative on $B_{\tau r}(x)$
and vanishes on $\partial B_{\tau r}(x)$. Now let us define $h^{i}:B_{\sqrt{\tau}r}(x)\to\mathbb{R}$
by 
\[
h^{i}(y):=\min\{u^{i}(y),\,\frac{s_{i}}{g(\sqrt{\tau}r)}\big(g(|y-x|)\big)^{+}\},
\]
where $s_{i}:=\max_{\overline{B}_{\sqrt{\tau}r}(x)}u^{i}$. We extend
$h^{i}$ by $u^{i}$ outside of $B_{\sqrt{\tau}r}(x)$. Note that
we have $h^{i}=0$ on $\{u^{i}=0\}\cap\overline{B}_{\tau r}(x)$ and
$h^{i}=u^{i}=\varphi^{i}$ on $\partial \Omega$. Hence $h^{i}$ competes
with $v^{i}$, and we have $\int_{\Omega}|\nabla v^{i}|^{p}\,dx\le\int_{\Omega}|\nabla h^{i}|^{p}\,dx$.
Therefore we can exchange $v^{i}$ by $h^{i}$ in the inequality (\ref{eq: 1})
to get 
\[
\frac{\varepsilon c_{0}}{C_{1}}\big|\{|\mathbf{u}|>0\}\cap B_{\tau r}(x)\big|\le\sum_{i\le m}\int_{B_{\sqrt{\tau}r}(x)}|\nabla h^{i}|^{p}-|\nabla u^{i}|^{p}\,dy.
\]
Now since $h^{i}=0$ on $B_{\tau r}(x)$, we can rewrite the above
inequality as 
\begin{align}
\frac{\varepsilon c_{0}}{C_{1}}\big|\{|\mathbf{u}|>0\}\cap B_{\tau r}(x)\big|\,+ & \sum_{i\le m}\int_{B_{\tau r}(x)}|\nabla u^{i}|^{p}\,dy\nonumber \\
 & \le\sum_{i\le m}\int_{B_{\sqrt{\tau}r}(x)-B_{\tau r}(x)}|\nabla h^{i}|^{p}-|\nabla u^{i}|^{p}\,dy.\label{eq: 2}
\end{align}
But 
\[
|\nabla h^{i}|^{p}-|\nabla u^{i}|^{p}\le-p|\nabla h^{i}|^{p-2}\nabla h^{i}\cdot\nabla(u^{i}-h^{i}),
\]
since for two vectors $a,b$ we have $|a|^{p}-|b|^{p}\le-p|a|^{p-2}a\cdot(b-a)$.
So we can estimate the right hand side of \eqref{eq: 2} as follows
(using integration by parts, and the facts that $\Delta_{p}h^{i}=0$
on $\{u^{i}>h^{i}\}$, $h^{i}=0$ on $\partial B_{\tau r}(x)$, and
$h^{i}=u^{i}$ on $\partial B_{\sqrt{\tau}r}(x)$): 
\begin{align*}
\int_{B_{\sqrt{\tau}r}(x)-B_{\tau r}(x)} & |\nabla h^{i}|^{p}-|\nabla u^{i}|^{p}\,dy\\
 & \le-p\int_{B_{\sqrt{\tau}r}(x)-B_{\tau r}(x)}|\nabla h^{i}|^{p-2}\nabla(u^{i}-h^{i})\cdot\nabla h^{i}\,dy\\
 & =p\int_{\partial B_{\tau r}(x)}(u^{i}-h^{i})|\nabla h^{i}|^{p-2}\nabla h^{i}\cdot\nu\,d\sigma\\
 & \qquad-p\int_{\partial B_{\sqrt{\tau}r}(x)}(u^{i}-h^{i})|\nabla h^{i}|^{p-2}\nabla h^{i}\cdot\nu\,d\sigma\\
 & =p\int_{\partial B_{\tau r}(x)}u^{i}|\nabla h^{i}|^{p-2}\nabla h^{i}\cdot\nu\,d\sigma=C(n,p,\tau)\frac{s_{i}^{p-1}}{r^{p-1}}\int_{\partial B_{\tau r}(x)}u^{i}\,d\sigma,
\end{align*}
where the last equality is calculated using the fact $h^{i}(y)=\frac{s_{i}}{g(\sqrt{\tau}r)}\big(g(|y-x|)\big)^{+}=0$
on $\overline{B}_{\tau r}(x)$; hence on $\partial B_{\tau r}(x)$
we have 
\[
\nabla h^{i}=C(\tau)s_{i}r^{\frac{n-p}{p-1}}\begin{cases}
\frac{2|p-n|}{p-1}|y-x|^{\frac{2-n-p}{p-1}}(y-x) & p\ne n,\\
|y-x|^{-2}(y-x) & p=n,
\end{cases}
\]
and thus 
\begin{align*}
|\nabla h^{i}|^{p-2}\nabla h^{i}\cdot\nu & =C(n,p,\tau)s_{i}^{p-1}r^{n-p}|y-x|^{2-n-p}|y-x|^{p-2}(y-x)\cdot\frac{(y-x)}{\tau r}\\
 & =C(n,p,\tau)s_{i}^{p-1}r^{n-p}|y-x|^{-n+2}\frac{1}{\tau r}=C(n,p,\tau)\frac{s_{i}^{p-1}}{r^{p-1}}.
\end{align*}
Hence (\ref{eq: 2}) becomes 
\begin{align}
\frac{\varepsilon c_{0}}{C_{1}}\big|\{|\mathbf{u}|>0\}\cap B_{\tau r}(x)\big|\,+ & \sum_{i\le m}\int_{B_{\tau r}(x)}|\nabla u^{i}|^{p}\,dy\nonumber \\
 & \le C(n,p,\tau)\sum_{i\le m}\frac{s_{i}^{p-1}}{r^{p-1}}\int_{\partial B_{\tau r}(x)}u^{i}\,d\sigma.\label{eq: 3}
\end{align}
On the other hand we have 
\begin{align}
\int_{\partial B_{\tau r}(x)}u^{i}\,d\sigma & \le c(n,\tau)\Big(\int_{B_{\tau r}(x)}u^{i}\,dy+\int_{B_{\tau r}(x)}|\nabla u^{i}|\,dy\Big)\nonumber \\
 & \le c(n,\tau)\Big((s_{i}+1)\cdot\big|\{|\mathbf{u}|>0\}\cap B_{\tau r}(x)\big|+\int_{B_{\tau r}(x)}|\nabla u^{i}|^{p}\,dy\Big),\label{eq: 4}
\end{align}
where in the last line we estimated $u^{i},|\nabla u^{i}|$ from above
by $s_{i},1+|\nabla u^{i}|^{p}$ on the set $\{u^{i}>0\}\subset\{|\mathbf{u}|>0\}$.
Next note that 
\begin{equation}
s_{i}=\max_{\overline{B}_{\sqrt{\tau}r}(x)}u^{i}\le\sup_{B_{r/2}(x)}u^{i}\le rm_{\varepsilon}(\tau),\label{eq: 5}
\end{equation}
since $\sqrt{\tau}<1/2$. Combining the inequalities (\ref{eq: 3}),
(\ref{eq: 4}), and (\ref{eq: 5}) we get 
\begin{align*}
 & \frac{\varepsilon c_{0}}{C_{1}}\big|\{|\mathbf{u}|>0\}\cap B_{\tau r}(x)\big|+\sum_{i\le m}\int_{B_{\tau r}(x)}|\nabla u^{i}|^{p}\,dy\\
 & \le cC\sum_{i\le m}\frac{s_{i}^{p-1}}{r^{p-1}}\Big((s_{i}+1)\cdot\big|\{|\mathbf{u}|>0\}\cap B_{\tau r}(x)\big|+\int_{B_{\tau r}(x)}|\nabla u^{i}|^{p}\,dy\Big)\\
 & \le cC\,m_{\varepsilon}^{p-1}(\tau)\Big(\big|\{|\mathbf{u}|>0\}\cap B_{\tau r}(x)\big|\sum_{i\le m}(s_{i}+1)+\sum_{i\le m}\int_{B_{\tau r}(x)}|\nabla u^{i}|^{p}\,dy\Big).
\end{align*}
Now if $m_{\varepsilon}(\tau)$ is small enough, we must necessarily
have $|\mathbf{u}|=0$ on $B_{\tau r}(x)$, as desired. 
\end{proof}
Now let us set 
\begin{align*}
U & :=\{x\in \Omega:|\mathbf{u}(x)|>0\},\\
E & :=\{x\in \Omega:|\mathbf{u}(x)|=0\}.
\end{align*}
\begin{lem}\label{lem:coperative}
For every $i$ we have 
\[
U=\{x\in\Omega:u^{i}(x)>0\},\qquad E=\{x\in\Omega:u^{i}(x)=0\}.
\]
\end{lem}

\begin{proof}
By Theorem \ref{thm: Lip-reg}, each $u^{i}$ is $p$-harmonic in
the open set $U$. So in each component of $U$ either $u^{i}>0$
or $u^{i}\equiv0$ (by the strong maximum principle). Now consider
a component of $U$, say $U_{1}$. If $\partial U_{1}$ does not intersect
$\partial\Omega$ then it must be a subset of $E$. Therefore every
$u^{i}$ vanishes on $\partial U_{1}$, and hence every $u^{i}$ vanishes
on $U_{1}$ by the maximum principle. So we would have $U_{1}\subset E$,
which is a contradiction. Thus $\partial U_{1}$ must intersect $\partial\Omega$.
Hence each $u^{i}>0$ on $U_{1}$, since they are positive on $\partial\Omega$.
Therefore each $u^{i}$ is positive on every component of $U$, as
desired.
\end{proof}

\begin{cor}\label{cor-nondegeneracy}
There are $c,C>0$ such that for $x\in U$ near $\partial E$ we have
\[
c\cdot\mathrm{dist}(x,\partial E)\le|\mathbf{u}(x)|\le C\cdot\mathrm{dist}(x,\partial E).
\]
\end{cor}

\begin{proof}
The right hand side inequality holds according to the Lipschitz regularity
of the solutions, Theorem \ref{thm: Lip-reg}. To see the left hand
side inequality, we argue indirectly. Assume to the contrary that there exists a sequence $x_{k}\in U$ such that 
\begin{equation}
|\mathbf{u}(x_{k})|\le\frac{1}{k}\mathrm{dist}(x_{k},\partial E).\label{cor:nondeg-Lip:eq1}
\end{equation}
Let $r_{k}=\mathrm{dist}(x_{k},\partial E)$ and define 
\[
\mathbf{u}_{k}(x)=\frac{\mathbf{u}(x_{k}+r_{k}x)}{r_{k}}.
\]
The sequence $\mathbf{u}_{k}$ is uniformly bounded and uniformly Lipschitz in $B_{1}$ due to Lipschitz regularity
of $\mathbf{u}$ and assumption \eqref{cor:nondeg-Lip:eq1}.

Recall that $\Delta_{p}u_{k}^{i}=0$ in $U$, then we
may choose a converging subsequence $\mathbf{u}_{k}\to\mathbf{u}_{0}$ such
that $u_{0}^{i}$ is also $p$-harmonic. Furthermore, by Theorem \ref{thm: Du>c}
we get that 
\[
\sup_{B_{1/2}(0)}|\mathbf{u}_{0}|=\lim_{k\to\infty}\sup_{B_{1/2}(0)}|\mathbf{u}_{k}|\ge m_{\varepsilon}>0,
\]
since $|\mathbf{u}_{k}(0)|>0$. Also, \eqref{cor:nondeg-Lip:eq1} yields that $\mathbf{u}_{0}(0)=0$, which
contradicts the maximum (minimum) principle; remember that each component
of $\mathbf{u}_{0}$ is nonnegative. 
\end{proof}
\begin{cor}
There exists $c=c_{\varepsilon}\in(0,1)$ such that for any $x\in\partial U$
and small enough $r$ we have 
\begin{equation}
c\le\frac{\big|E\cap B_{r}(x)\big|}{\big|B_{r}(x)\big|}\le1-c.\label{thm:porosity-rel1}
\end{equation}
\end{cor}

\begin{proof}
The proof is similar to the proof of Theorem 4.2 in \cite{fotouhi2023minimization}.
By Theorem \ref{thm: Du>c}, there exists $z\in B_{r/2}(x)$ such
that $|\mathbf{u}(z)|\ge m_{\varepsilon}r>0$. Now for any $y\in B_{\tau r}(z)$
we have 
\[
|\mathbf{u}(y)-\mathbf{u}(z)|\le\mathrm{Lip}(\mathbf{u})|y-z|<\mathrm{Lip}(\mathbf{u})\tau r<\frac{m_{\varepsilon}r}{2},
\]
provided that $\tau$ is small enough. Hence we must have $|\mathbf{u}(y)|>\frac{m_{\varepsilon}r}{2}>0$.
This gives the upper estimate in \eqref{thm:porosity-rel1}.

To prove the estimate from below, suppose to the contrary that there
exists a sequence of points $x_{k}\in\partial U$ and radii $r_{k}\to0$
such that 
\[
\big|\{|\mathbf{u}|=0\}\cap B_{r_{k}}(x_{k})\big|<\frac{1}{k}|B_{r_{k}}(x_{k})|=\frac{1}{k}r_{k}^{n}\,|B_{1}|.
\]
Now let us define 
\[
\mathbf{u}_{k}(x)=\frac{\mathbf{u}(x_{k}+r_{k}x)}{r_{k}}.
\]
Note that $\mathbf{u}_{k}(0)=\mathbf{u}(x_{k})=0$, and thus $\mathbf{u}_{k}$
is uniformly bounded and uniformly Lipschitz in $B_{1}=B_{1}(0)$
due to Lipschitz regularity of $\mathbf{u}$. Also 
\[
\big|\{|\mathbf{u}_{k}|=0\}\cap B_{1}\big|=\frac{1}{r_{k}^{n}}\big|\{|\mathbf{u}|=0\}\cap B_{r_{k}}(x_{k})\big|\underset{k\to\infty}{\longrightarrow}0.
\]
Let $v_{k}^{i}$ be a $p$-harmonic function in $B_{1/2}$ with boundary
data $v_{k}^{i}=u_{k}^{i}$ on $\partial B_{1/2}$. Then $h_{k}^{i}(x)=r_{k}v_{k}^{i}(\frac{x-x_{k}}{r_{k}})$
is a $p$-harmonic function in $B_{r_{k}/2}(x_{k})$ with boundary
data $h_{k}^{i}=u^{i}$ on $\partial B_{r_{k}/2}(x_{k})$. Now, similarly
to the proof of Theorem \ref{thm: Du<C}, we can show that 
\begin{align}
\int_{B_{1/2}}|\nabla(u_{k}^{i}-v_{k}^{i})|^{p}\,dx & =\frac{1}{r_{k}^{n}}\int_{B_{r_{k}/2}(x_{k})}|\nabla(u^{i}-h_{k}^{i})|^{p}\,dx\label{thm:porosity-rel2}\\
 & \le\frac{C}{\varepsilon}\frac{1}{r_{k}^{n}}\big|\{|\mathbf{u}|=0\}\cap B_{r_{k}}(x_{k})\big|\underset{k\to\infty}{\longrightarrow}0.\nonumber 
\end{align}
(Note that the constant $C$ does not depend on the radius $r_{k}$
or the point $x_{k}$.) 

Since $u_{k}^{i}$ and therefore $v_{k}^{i}$ are uniformly Lipschitz
in $B_{1/4}$, we may assume that $u_{k}^{i}\to u_{0}^{i}$ and $v_{k}^{i}\to v_{0}^{i}$
uniformly in $B_{1/4}$. Observe that $\Delta_{p}v_{0}^{i}=0$, and
\eqref{thm:porosity-rel2} implies that $u_{0}^{i}=v_{0}^{i}+C$ for
some constant $C$. Thus $\Delta_{p}u_{0}^{i}=0$ in $B_{1/4}$ and
from the strong minimum principle it follows $u_{0}^{i}\equiv0$ in
$B_{1/4}$, since $u_{0}^{i}\ge0$ and $u_{0}^{i}(0)=\lim u_{k}^{i}(0)=0$.
On the other hand the nondegeneracy property, Theorem \ref{thm: Du>c},
implies that (since $x_{k}$ is not in the interior of $\{|\mathbf{u}|=0\}$)
\[
\|\mathbf{u}_{k}\|_{L^{\infty}(B_{1/4})}=\frac{1}{r_{k}}\|\mathbf{u}\|_{L^{\infty}(B_{r_{k}/4}(x_{k}))}\ge m_{\varepsilon}/2>0.
\]
Therefore we get $\|\mathbf{u}_{0}\|_{L^{\infty}(B_{1/4})}\ge m_{\varepsilon}/2$,
which is a contradiction.
\end{proof}
Hence we can apply the results in section 4 of \cite{alt1981existence}
and in section 3 of  \cite{alt1984free}  to conclude (see also sections 5 and 6 of \cite{fotouhi2023minimization}) 
\begin{thm} \label{Thm:Hausdorff-dim-FB}
Let $\mathbf{u}=\mathbf{u}_{\varepsilon}$ be a minimizer of $J_{\varepsilon}$
over $V$. Then we have 
\begin{enumerate}
\item The $(n-1)$-dimensional Hausdorff measure of $\partial E$ is locally
finite, i.e. $\mathcal{H}^{n-1}(\Omega'\cap\partial E)<\infty$ for
every $\Omega'\subset\subset \Omega$. Moreover, there exist positive constants
$c_{\varepsilon},C_{\varepsilon}$, depending on $n,p,\Omega,\Omega',\varepsilon$,
such that for each ball $B_{r}(x)\subset\Omega'$ with $x\in\partial E$
we have 
\[
c_{\varepsilon}r^{n-1}\le\mathcal{H}^{n-1}(B_{r}(x)\cap\partial E)\le C_{\varepsilon}r^{n-1}.
\]
\item There exist Borel functions $q^{i}=q_{\varepsilon}^{i}$
such that \[\Delta_{p}u^{i}=q^{i}\,\mathcal{H}^{n-1}\mres\partial E,\]that is, for any $\zeta\in C_{0}^{\infty}(\Omega)$ we have 
\[
-\int_{\Omega}A[u^{i}]\cdot\nabla\zeta\,dy=\int_{\partial E}\zeta q^{i}\,d\mathcal{H}^{n-1}.
\]

\item For $\mathcal{H}^{n-1}$-a.e. points $x\in\partial E$ we have 
\[
c_{\varepsilon}\le \sum_{i=1}^{m} q^{i}(x)\le C_{\varepsilon}.
\]
\item For $\mathcal{H}^{n-1}$-a.e. points $x\in\partial E$
an outward unit normal $\nu=\nu_{E}(x)$ is defined, and 
\[
u^{i}(x+y)=(q^{i}(x))^{\frac{1}{p-1}}(y\cdot\nu)^{+}+o(|y|),
\]
which allows us to define
$A_{\nu}u^{i}(x)= q^{i}(x)$ at those points. 
\item The reduced boundary $\partial_{\mathrm{red}}E$ satisfies $\mathcal{H}^{n-1}(\partial E-\partial_{\mathrm{red}}E)=0$.
\end{enumerate}
\end{thm}

\medskip

\section{The Original Problem}

In this section we will show that for $\varepsilon>0$ small enough,
a minimizer of $J_{\varepsilon}$ over $V$ satisfies $\big|\{|\mathbf{u}_{\varepsilon}|>0\}\big|=1$,
and hence it can be regarded as a solution to our original problem
(\ref{eq: main_eq}). Remember that 
\[
U=U_{\varepsilon}=\{|\mathbf{u}_{\varepsilon}|>0\},\qquad E=E_{\varepsilon}=\{|\mathbf{u}_{\varepsilon}|=0\}.
\]
Note that by Lemma \ref{lem: u>0 near bdry}, the free boundary $\partial E$
has a positive distance from the fixed boundary $\partial\Omega$.
We say $x\in\partial E$ is a \textit{regular point} of the free boundary
if it satisfies (3) and (4) in Theorem \ref{Thm:Hausdorff-dim-FB}%
. The set of such regular points of the free boundary will be denoted
by $\mathcal{R}=\mathcal{R}_{\varepsilon}$; Theorem \ref{Thm:Hausdorff-dim-FB}
shows that $\mathcal{H}^{n-1}(\partial E-\mathcal{R})=0$.
\begin{lem}
\label{lem: q<C}There is a constant $C>0$, independent of $\varepsilon$,
such that 
\[
\inf_{\mathcal{R}_{\varepsilon}}\Big(\sum_{i\le m}q_{\varepsilon}^{i}\Big)\le C.
\]
\end{lem}

\begin{rem*}
Note that $\sum_{i\le m}q_{\varepsilon}^{i}\ge c_{\varepsilon}>0$
by Theorem \ref{Thm:Hausdorff-dim-FB}. 
{} 
\end{rem*}
\begin{proof}
Let $\Omega'\subset\subset\Omega$ be a smooth open set with $|\Omega-\Omega'|=1$.
Let $\mathbf{u}_{0}$ be a vector-valued function on $\Omega-\Omega'$
that satisfies the equation $\Delta_{p}u_{0}^{i}=0$, and takes the
boundary values $\boldsymbol{\varphi}$ on $\partial\Omega$ and $0$
on $\partial\Omega'$. Then for some small enough $\delta_{0}$ we
have $\mathbf{u}_{0}\in V_{\delta_{0}}\subset V$; hence 
\begin{align*}
C=\int_{\partial\Omega}\Gamma(x,A_{\nu}\mathbf{u}_{0})\,d\sigma+1 & =J_{\varepsilon}(\mathbf{u}_{0})\ge J_{\varepsilon}(\mathbf{u}_{\varepsilon})\\
 & =\int_{\partial\Omega}\Gamma(x,A_{\nu}\mathbf{u}_{\varepsilon})\,d\sigma+f_{\varepsilon}\big(\big|\{|\mathbf{u}_{\varepsilon}|>0\}\big|\big)\\
 & \ge\int_{\partial\Omega}\sum_{i=1}^{m}\psi_{i}A_{\nu}u_{\varepsilon}^{i}-C\,d\sigma+f_{\varepsilon}\big(\big|\{|\mathbf{u}_{\varepsilon}|>0\}\big|\big)\\
 & \ge C\sum_{i=1}^{m}\int_{\Omega}|\nabla u_{\varepsilon}^{i}|^{p}\,dx-C+f_{\varepsilon}\big(\big|\{|\mathbf{u}_{\varepsilon}|>0\}\big|\big)\\
 & \ge-C+\frac{1}{\varepsilon}\big(\big|\{|\mathbf{u}_{\varepsilon}|>0\}\big|-1\big),
\end{align*}
where we have used (\ref{eq: Gamma coercive}) and Lemma \ref{lem: L2 bd der}.
Thus we get the bound
\[
|U|=\big|\{|\mathbf{u}_{\varepsilon}|>0\}\big|\le1+C\varepsilon.
\]
Note that $J_{\varepsilon}(\mathbf{u}_{0})$, and thus $C$, does
not depend on $\varepsilon$ due to the definition of $f_{\varepsilon}$.
As a result, we have a lower bound for the volume of $E$. Hence,
by the isoperimetric inequality, we have a lower bound for $\mathcal{H}^{n-1}(\partial E)$,
independent of $\varepsilon$. Now note that (keep in mind that $\nu_{E}$
points to the interior of $U$) 
\[
\int_{\partial\Omega}A_{\nu}u_{\varepsilon}^{i}\,d\sigma-\int_{\partial E}A_{\nu}u_{\varepsilon}^{i}\,d\mathcal{H}^{n-1}=\int_{U}\Delta_{p}u_{\varepsilon}^{i}\,dx=0.
\]
Therefore we get 
\begin{align*}
\int_{\partial E}A_{\nu}u_{\varepsilon}^{i}\,d\mathcal{H}^{n-1}=\int_{\partial\Omega}A_{\nu}u_{\varepsilon}^{i}\,d\sigma & =\int_{\Omega}\Delta_{p}u_{\varepsilon}^{i}\,dx\\
 & \le C+C\int_{\partial\Omega}\psi_{i}A_{\nu}u_{\varepsilon}^{i}\,d\sigma,
\end{align*}
where the last inequality follows from the remark below Lemma \ref{lem: L2 bd der}.
Thus we have 
\begin{align*}
\inf_{\mathcal{R}_{\varepsilon}}\Big(\sum_{i\le m}A_{\nu}u_{\varepsilon}^{i}\Big)\mathcal{H}^{n-1}(\partial E) & \le\int_{\partial E}\sum_{i\le m}A_{\nu}u_{\varepsilon}^{i}\,d\mathcal{H}^{n-1}\\
 & \le C+C\int_{\partial\Omega}\sum_{i\le m}\psi_{i}A_{\nu}u_{\varepsilon}^{i}\,d\sigma\\
 & \le C+C\int_{\partial\Omega}\Gamma\big(x,A_{\nu}\mathbf{u}_{\varepsilon}\big)\,d\sigma\tag{by \eqref{eq: Gamma coercive}}\\
 & \le C+CJ_{\varepsilon}(\mathbf{u}_{\varepsilon})\le C+CJ_{\varepsilon}(\mathbf{u}_{0})\le C,
\end{align*}
which gives the desired (noting that $q^{i}=A_{\nu}u_{\varepsilon}^{i}$ by Theorem \ref{Thm:Hausdorff-dim-FB}). 
\end{proof}
\begin{lem}
\label{lem: vol >1}For small enough $\varepsilon$ we have 
\[
\big|\{|\mathbf{u}_{\varepsilon}|>0\}\big|\ge1.
\]
\end{lem}

\begin{proof}
Consider a point $z_{0}\in\Omega^{c}$ which has distance $\delta_{0}$
from $\partial\Omega$. Then the ball $B_{\delta_{0}}(z_{0})$ is
an exterior tangent ball to $\partial\Omega$. Let $t=t(\varepsilon)$
be the first time at which $\partial B_{\delta_{0}+t}(z_{0})$ intersects
$\partial\{|\mathbf{u}_{\varepsilon}|=0\}$, at a point $x_{0}=x_{0}(\varepsilon)$.
Now let $v$ be a $p$-harmonic function in $B_{\delta_{0}+t}(z_{0})-\overline{B}_{\delta_{0}}(z_{0})$
with boundary values $0$ on $\partial B_{\delta_{0}+t}(z_{0})$ and
$c_{0}$ on $\partial B_{\delta_{0}}(z_{0})$, where $c_{0}=\min_{i}\min_{\partial\Omega}\varphi^{i}>0$.
Then on $\partial\big(\Omega\cap B_{\delta_{0}+t}(z_{0})\big)$ we
have $v\le u^{i}$; so by the maximum principle we have $v\le u^{i}$
in $\Omega\cap B_{\delta_{0}+t}(z_{0})$. However, by an easy modification
of the proof of Hopf's lemma (Lemma \ref{lem: Hopf}), we can see
that 
\[
v(x)\ge cc_{0}\,\mathrm{dist}\big(x,\partial B_{\delta_{0}+t}(z_{0})\big),
\]
where the constant $c$ only depends on $n,p,\delta_{0}$. Therefore,
for points $x$ in the line segment between $x_{0},z_{0}$ we have
\[
u^{i}(x)\ge v(x)\ge cc_{0}\,\mathrm{dist}\big(x,\partial B_{\delta_{0}+t}(z_{0})\big)=cc_{0}|x-x_{0}|.
\]
Now consider the ball $B_{r}(x_{0})$ for small enough $r$. Then
we have 
\[
\frac{1}{r}\sup_{B_{r/2}(x_{0})}u^{i}\ge\frac{1}{r}cc_{0}\frac{r}{2}=\frac{cc_{0}}{2},
\]
independently of $\varepsilon$. 

Let $\mathbf{h}$ be the vector-valued function which satisfies $\Delta_{p}h^{i}=0$
in $B_{r}(x_{0})$, and is equal to $\mathbf{u}$ in $\Omega-B_{r}(x_{0})$.
By Lemma \ref{lem: |u is 0| bdd} and the fact that $h^{i}\ge u^{i}$
we have 
\begin{align*}
\int_{B_{r}(x_{0})}|\nabla(u^{i}-h^{i})|^{p}\,dy & \ge C\Big(\frac{1}{r}\sup_{B_{r/2}(x_{0})}h^{i}\Big)^{p}\cdot\big|B_{r}(x_{0})\cap\{u^{i}=0\}\big|\\
 & \ge C\Big(\frac{1}{r}\sup_{B_{r/2}(x_{0})}u^{i}\Big)^{p}\cdot\big|B_{r}(x_{0})\cap\{u^{i}=0\}\big|\\
 & \ge C\big|B_{r}(x_{0})\cap\{u^{i}=0\}\big|\ge C\big|B_{r}(x_{0})\cap\{|\mathbf{u}|=0\}\big|.
\end{align*}
Next let $\mathbf{v}$ be the function given by Lemma \ref{lem: aux obst pr}
for $B_{r}(x_{0})$. We know that $J_{\varepsilon}(\mathbf{u})\le J_{\varepsilon}(\mathbf{v})$.
Then similarly to the proof of Theorem \ref{thm: Du<C} we can see
that 
\begin{align*}
C\sum_{i\le m}\int_{B_{r}(x_{0})}|\nabla(u^{i}-h^{i})|^{p}\,dy & \le\int_{\partial\Omega}\Gamma(x,A_{\nu}\mathbf{u})-\Gamma(x,A_{\nu}\mathbf{v})\,d\sigma\\
 & \le f_{\varepsilon}\big(\big|\{|\mathbf{v}|>0\}\big|\big)-f_{\varepsilon}\big(\big|\{|\mathbf{u}|>0\}\big|\big).
\end{align*}
(A closer inspection of the proof of Theorem \ref{thm: Du<C} reveals
that the constant $C$ in the above estimate only depends on $n,p,\Omega,\boldsymbol{\varphi},\Gamma$.) 

Now suppose to the contrary that $\big|\{|\mathbf{u}|>0\}\big|<1$.
Then, since $0\le u^{i}\le v^{i}$, and outside of $B_{r}(x_{0})$,
$|\mathbf{u}|=0$ implies $|\mathbf{v}|=0$, we have 
\[
\big|\{|\mathbf{v}|>0\}\big|\le\big|\{|\mathbf{u}|>0\}\big|+\big|B_{r}(x_{0})\cap\{|\mathbf{u}|=0\}\big|<1
\]
for small enough $r$. Hence (using the monotonicity of $f_{\varepsilon}$)
we have 
\begin{align*}
 & f_{\varepsilon}\big(\big|\{|\mathbf{v}|>0\}\big|\big)-f_{\varepsilon}\big(\big|\{|\mathbf{u}|>0\}\big|\big)\\
 & \qquad\qquad\le f_{\varepsilon}\Big(\big|\{|\mathbf{u}|>0\}\big|+\big|B_{r}(x_{0})\cap\{|\mathbf{u}|=0\}\big|\Big)-f_{\varepsilon}\big(\big|\{|\mathbf{u}|>0\}\big|\big)\\
 & \qquad\qquad=\varepsilon\big|B_{r}(x_{0})\cap\{|\mathbf{u}|=0\}\big|.
\end{align*}
Combining this estimate with the estimates of the above paragraph,
and using (\ref{thm:porosity-rel1}), we obtain 
\[
0<C\big|B_{r}(x_{0})\cap\{|\mathbf{u}|=0\}\big|\le\varepsilon\big|B_{r}(x_{0})\cap\{|\mathbf{u}|=0\}\big|,
\]
which gives a positive lower bound for $\varepsilon$, and results
in a contradiction. 
\end{proof}
\begin{thm}
\label{thm: vol is 1}
When $\varepsilon$ is small enough we have
\[
\big|\{|\mathbf{u}_{\varepsilon}|>0\}\big|=1.
\]
\end{thm}

\begin{proof}
By the above lemma we only need to show that $\big|\{|\mathbf{u}_{\varepsilon}|>0\}\big|\le1$.
To this end, we will compare $\mathbf{u}_{\varepsilon}$ with a suitable
perturbation of itself. Let $x_{0}\in\mathcal{R}$, and let $\rho:\mathbb{R}\to\mathbb{R}$
be a nonnegative smooth function supported in $(0,1)$. For small
enough $r,\lambda>0$ we consider the vector field 
\[
T_{r}(x):=\begin{cases}
x+r\lambda\rho(|x-x_{0}|/r)\nu(x_{0}) & \textrm{if }x\in B_{r}(x_{0}),\\
x & \text{elsewhere}.
\end{cases}
\]
Here, $\nu(x_{0})$ is the outward normal vector provided in (4) of
Theorem \ref{Thm:Hausdorff-dim-FB}. We can easily see that for $x$
in $B_{r}(x_{0})$ we have 
\begin{equation}
DT_{r}(x)\,\cdot=I\cdot+\,\lambda\rho'(|x-x_{0}|/r)\frac{\langle x-x_{0},\cdot\,\rangle}{|x-x_{0}|}\nu(x_{0}),\label{eq:DT_r-2}
\end{equation}
where $I$ is the identity matrix. Hence, if $\lambda$ is small enough,
$T_{r}$ is a diffeomorphism that maps $B_{r}(x_{0})$ onto itself.

Now consider 
\[
\mathbf{v}_{r}(x):=\mathbf{u}(T_{r}^{-1}(x))
\]
for $r>0$ small enough. Similarly to the proof of Theorem \ref{thm: Du>c},
we consider the vector-valued function $\mathbf{w}$ whose components
minimize the Dirichlet $p$-energy subject to the condition 
\[
w^{i}\le0\;\textrm{ on }\;\{\mathbf{u}=0\}\cup\big(\overline{B}_{r}(x_{0})\cap\{\mathbf{v}_{r}=0\}\big).
\]
With a calculation similar to (\ref{boundary-to-gradient-v<u}) and
(\ref{eq: Gamma > sum}) we get 
\begin{align}
0\le J_{\varepsilon}(\mathbf{w})-J_{\varepsilon}(\mathbf{u}) & \le C\sum_{i=1}^{m}\int_{\Omega}|\nabla w^{i}|^{p}-|\nabla u^{i}|^{p}\,dx\nonumber \\
 & \qquad\qquad+f_{\varepsilon}\big(\big|\{|\mathbf{w}|>0\}\big|\big)-f_{\varepsilon}\big(\big|\{|\mathbf{u}|>0\}\big|\big)\nonumber \\
 & \le C\sum_{i=1}^{m}\int_{B_{r}(x_{0})}|\nabla v_{r}^{i}|^{p}-|\nabla u^{i}|^{p}\,dx\label{inq:compare u-w2-1}\\
 & \qquad\qquad+f_{\varepsilon}\big(\big|\{|\mathbf{w}|>0\}\big|\big)-f_{\varepsilon}\big(\big|\{|\mathbf{u}|>0\}\big|\big),\nonumber 
\end{align}
where in the last inequality we have compared the Dirichlet $p$-energy
of $\mathbf{w}$ with that of $\mathbf{v}_{r}\chi_{B_{r}(x_{0})}+\mathbf{u}\chi_{\Omega-B_{r}(x_{0})}$. 

Now notice that 
\begin{align*}
\int_{B_{r}(x_{0})}|\nabla v_{r}^{i}|^{p}\,dx & =\int_{B_{r}(x_{0})}\big|DT_{r}(T_{r}^{-1}(x))^{-1}\nabla u^{i}(T_{r}^{-1}(x))\big|^{p}\,dx\\
 & =\int_{B_{r}(x_{0})}\big|DT_{r}(y)^{-1}\nabla u^{i}(y)\big|^{p}\,|\det DT_{r}(y)|\,dy\\
 & =r^{n}\int_{B_{1}}\big|DT_{r}(y)^{-1}\nabla u^{i}(y)\big|^{p}\,|\det DT_{r}(y)|\,dz.\tag{\ensuremath{z=\frac{y-x_{0}}{r}}}
\end{align*}
From (\ref{eq:DT_r-2}), for small enough $\lambda$ we can write
\begin{align}
DT_{r}(y)^{-1} & =I+\Big(\sum_{k=1}^{\infty}(-1)^{k}\lambda^{k}\rho'(|z|)^{k}\frac{\langle z,\nu\rangle^{k-1}}{|z|^{k-1}}\Big)\frac{\langle z,\cdot\,\rangle}{|z|}\nu(x_{0})\nonumber \\
 & =I-\lambda\rho'(|z|)\frac{\langle z,\cdot\,\rangle}{|z|}\nu(x_{0})+\lambda^{2}g(\lambda,z)\frac{\langle z,\cdot\,\rangle}{|z|}\nu(x_{0})\label{eq: DT^-1-1}
\end{align}
for some $g$. Hence we have 
\[
DT_{r}(y)^{-1}\nabla u^{i}(y)=\nabla u^{i}(y)-\lambda\rho'(|z|)\frac{\langle z,\nabla u^{i}(y)\rangle}{|z|}\nu(x_{0})+O(\lambda^{2}).
\]
Thus 
\[
\big|DT_{r}(y)^{-1}\nabla u^{i}(y)\big|^{2}=|\nabla u^{i}(y)|^{2}-2\lambda\rho'(|z|)\frac{\langle z,\nabla u^{i}(y)\rangle}{|z|}\langle\nu(x_{0}),\nabla u^{i}(y)\rangle+O(\lambda^{2}),
\]
and therefore 
\begin{align*}
\big|DT_{r}(y)^{-1}\nabla u^{i}(y)\big|^{p} & =|\nabla u^{i}(y)|^{p}\Big(1-p\lambda\rho'(|z|)\frac{\langle z,\nabla u^{i}(y)\rangle}{|z||\nabla u^{i}(y)|^{2}}\langle\nu(x_{0}),\nabla u^{i}(y)\rangle\Big)\\
 & \qquad+O(\lambda^{2}).
\end{align*}
Also, we have (noting that $DT_{r}$ is the identity matrix plus a
rank $1$ matrix) 
\[
|\det DT_{r}(y)|=1+\lambda\rho'(|z|)\frac{\langle z,\nu(x_{0})\rangle}{|z|}.
\]
All these together, we obtain (remember that $y=x_{0}+rz$) 
\begin{align*}
 & r^{-n}\int_{B_{r}(x_{0})}|\nabla v_{r}^{i}|^{p}-|\nabla u^{i}|^{p}\,dx\\
 & \qquad=\lambda\int_{B_{1}}|\nabla u^{i}(y)|^{p}\rho'(|z|)\left(\frac{\langle z,\nu(x_{0})\rangle}{|z|}-p\frac{\langle z,\nabla u^{i}(y)\rangle\langle\nabla u^{i}(y),\nu(x_{0})\rangle}{|z||\nabla u^{i}(y)|^{2}}\right)dz\\
 & \qquad\quad+O(\lambda^{2}).
\end{align*}
Now consider the blowup sequence $\mathbf{u}_{r}(z):=\mathbf{u}(x_{0}+rz)/r$.
We know that as $r\to0$ (see \cite{alt1984free})
\[
\begin{split} & \{u_{r}^{i}>0\}\cap B_{1}\to\{z:z\cdot\nu(x_{0})>0\}\cap B_{1},\\
 & \nabla u^{i}(y)=\nabla u_{r}^{i}(z)\to(q^{i}(x_{0}))^{\frac{1}{p-1}}\nu(x_{0})\chi_{\{z\cdot\nu(x_{0})>0\}},\qquad\text{a.e. in }B_{1}.
\end{split}
\]
Therefore we get 
\begin{align*}
 & r^{-n}\int_{B_{r}(x_{0})}|\nabla v_{r}^{i}|^{p}-|\nabla u^{i}|^{p}\,dx\\
 & \qquad\underset{r\to0}{\longrightarrow}-(p-1)\lambda|q^{i}(x_{0})|^{\frac{p}{p-1}}\int_{B_{1}\cap\{z\cdot\nu(x_{0})>0\}}\rho'(|z|)\frac{\langle z,\nu(x_{0})\rangle}{|z|}\,dz+O(\lambda^{2}).
\end{align*}
Note that the formula (\ref{eq: DT^-1-1}) for $(DT_{r})^{-1}$ does
not depend on $r$, and the function $\cdot\mapsto|\cdot|^{p}$ is
continuous; so the $O(\lambda^{2})$ term converges to an $O(\lambda^{2})$
term as $r\to0$. Next note that 
\[
\mathrm{div}(\rho(|z|)\nu)=\frac{\rho'(|z|)}{|z|}\langle z,\nu\rangle.
\]
Thus (noting that $\rho(|z|)$ is zero near $\partial B_{1}$) 
\begin{align*}
\int_{B_{1}\cap\{z\cdot\nu(x_{0})>0\}}\rho'(|z|)\frac{\langle z,\nu(x_{0})\rangle}{|z|}\,dz & =-\int_{B_{1}\cap\{z\cdot\nu(x_{0})=0\}}\rho(|z|)\,dz\\
 & =-\omega_{n-1}\int_{0}^{1}\rho(t)t^{n-1}\,dt=-C_{\rho}\omega_{n-1},
\end{align*}
where $\omega_{n-1}$ is the volume of the $(n-1)$-dimensional ball
of radius $1$, and $C_{\rho}$ depends only on $\rho$. Hence we
can write 
\begin{align*}
 & \int_{B_{r}(x_{0})}|\nabla v_{r}^{i}|^{p}-|\nabla u^{i}|^{p}\,dx\\
 & \qquad=\big[(p-1)\lambda C_{\rho}\omega_{n-1}|q^{i}(x_{0})|^{\frac{p}{p-1}}+O(\lambda^{2})\big]r^{n}+o(r^{n}).
\end{align*}

On the other hand, 
\[
\begin{split}
& \lim_{r\to 0} r^{-n}\big|B_{r}(x_{0})\cap\{|\mathbf{v}_{r}|>0\}\big|  = \lim_{r\to 0} r^{-n}\int_{\{|\mathbf{v}_{r}|>0\}\cap B_{r}(x_{0})}dx \\
 & \qquad= \lim_{r\to 0} r^{-n}\int_{\{|\mathbf{u}|>0\}\cap B_{r}(x_{0})}|\det DT_{r}(y)|\,dy \\
 & \qquad= \int_{B_{1}\cap\{z\cdot\nu(x_{0})>0\}}1+\lambda\rho'(|z|)\frac{\langle z,\nu(x_{0})\rangle}{|z|}\,dz \\
 & \qquad= \frac{1}{2}\omega_{n}-\lambda\omega_{n-1}\int_{0}^{1}\rho(t)t^{n-1}\,dt = \frac{1}{2}\omega_{n}-\lambda C_{\rho}\omega_{n-1}.
\end{split}
\]
Thus for $A_{0}:=\big(\{|\mathbf{u}|>0\}-B_{r}(x_{0})\big)\cup\big(\{|\mathbf{v}_{r}|>0\}\cap B_{r}(x_{0})\big)$
we have 
\begin{align*}
|A_{0}|-\big|\{|\mathbf{u}|>0\}\big| & =\big|B_{r}(x_{0})\cap\{|\mathbf{v}_{r}|>0\}\big|-\big|B_{r}(x_{0})\cap\{|\mathbf{u}|>0\}\big|\\
 & =-\lambda C_{\rho}\omega_{n-1}r^{n}+o(r^{n}).
\end{align*}
In addition, it is easy to see that $\{|\mathbf{w}|>0\}\subset A_{0}$.

Now suppose to the contrary that $\big|\{|\mathbf{u}|>0\}\big|>1$.
Then we can choose $r$ small enough so that 
\[
|A_{0}|=\big|\{|\mathbf{u}|>0\}\big|-\lambda C_{\rho}\omega_{n-1}r^{n}+o(r^{n})>1.
\]
Therefore, using the monotonicity of $f_{\varepsilon}$ we get 
\[
\begin{split}f_{\varepsilon}\big(\big|\{|\mathbf{w}|>0\}\big|\big) & -f_{\varepsilon}\big(\big|\{|\mathbf{u}|>0\}\big|\big) \le f_{\varepsilon}(|A_{0}|)-f_{\varepsilon}\big(\big|\{|\mathbf{u}|>0\}\big|\big)\\
 & =\frac{1}{\varepsilon}\big(|A_{0}|-\big|\{|\mathbf{u}|>0\}\big|\big)=-\frac{1}{\varepsilon}\lambda C_{\rho}\omega_{n-1}r^{n}+o(r^{n}).
\end{split}
\]
Finally, by putting all these estimates in (\ref{inq:compare u-w2-1}),
we obtain 
\begin{align*}
0 & \le C\sum_{i=1}^{m}\int_{B_{r}(x_{0})}|\nabla v_{r}^{i}|^{p}-|\nabla u^{i}|^{p}\,dx+f_{\varepsilon}(|A_{0}|)-f_{\varepsilon}\big(\big|\{|\mathbf{u}|>0\}\big|\big)\\
 & =\big[(p-1)\lambda C_{\rho}\omega_{n-1}\sum_{i=1}^{m}|q^{i}(x_{0})|^{\frac{p}{p-1}}+O(\lambda^{2})\big]r^{n}-\frac{1}{\varepsilon}\lambda C_{\rho}\omega_{n-1}r^{n}+o(r^{n}).
\end{align*}
Dividing by $r^{n}$ and letting $r\to0$, and then dividing by $\lambda$
and letting $\lambda\to0$, we get 
\[
\frac{1}{\varepsilon}\le(p-1)\sum_{i=1}^{m}|q^{i}(x_{0})|^{\frac{p}{p-1}}.
\]
Now if we choose $x_{0}$ such that 
\[
\sum_{i\le m}q^{i}(x_{0})\le\inf_{\mathcal{R}_{\varepsilon}}\Big(\sum_{i\le m}q^{i}\Big)+1,
\]
then by Lemma \ref{lem: q<C} (and the equivalence of all norms on
the finite-dimensional space $\mathbb{R}^{m}$) we have $\sum_{i\le m}|q^{i}(x_{0})|^{\frac{p}{p-1}}\le C$,
independently of $\varepsilon$. However, this implies that $\varepsilon$
has a positive lower bound, which is a contradiction.
\end{proof}

\section{Regularity of the free boundary (case $p=2$)}

We are going to show that $\mathcal{R}$ is an analytic hypersurface when $p=2$. To see this, we first derive the free boundary condition, also known as the optimality condition, in the following lemma. We perturb the optimal set $\Omega$ and compute the first variation of the energy functional $J_{\varepsilon}$. To perform this computation, it is
crucial to ensure that the $p$-harmonic solution within the perturbed domain is differentiable with respect to the perturbation parameter. When $p=2$, this can be established through the implicit function theorem . However, it is noteworthy that for $p\ne2$ the proof does not hold, primarily due to the ill-posedness of the derivative of the map $u\mapsto\Delta_{p}u$.



\begin{lem}\label{lem:FBC}
Let $\mathbf{u}$ be a solution of the minimization problem
(\ref{eq: main_eq}) for $p=2$. Let $h^{i}$ be the solution of 
\[
\begin{cases}
\Delta h^{i}=0 & \text{in }\Omega-E,\\
h^{i}=0 & \text{on }E,\\
h^{i}=\partial_{\xi_{i}}\Gamma(x,\partial_{\nu}\mathbf{u}) & \text{on }\partial\Omega.
\end{cases}
\]
Then, on the regular part of the free boundary we have 
\begin{equation}\label{FB-condition}
 \sum_{i=1}^{m}\partial_{\nu}h^{i}\partial_{\nu}u^{i} = C   
\end{equation}
for some positive constant $C$.
\end{lem}

\begin{proof}
Let $x_{1}$ and $x_{2}$ be two regular points in $\mathcal{R}$
with corresponding unit normal vectors $\nu(x_{1})$ and $\nu(x_{2})$.
Also, let $\rho:\mathbb{R}\to\mathbb{R}$ be a nonnegative smooth
function supported in $(0,1)$. Similarly to the proof of Theorem
\ref{thm: vol is 1} we define the vector field 
\[
T_{r,\lambda}(x):=\begin{cases}
x-r\lambda\rho(|x-x_{1}|/r)\nu(x_{1}) & \textrm{if }x\in B_{r}(x_{1}),\\
x+r\lambda\rho(|x-x_{2}|/r)\nu(x_{2}) & \textrm{if }x\in B_{r}(x_{2}),\\
x & \text{elsewhere},
\end{cases}
\]
for small enough $r,\lambda>0$ (which makes $T_{r,\lambda}$ a diffeomorphism from $B_r(x_a)$ onto itself for $a=1,2$).

Now for some fixed $r>0$ let $E_{\lambda}=T_{r,\lambda}^{-1}(E)$, and assume that $\mathbf{w}_{\lambda}$
solves 
\[
\begin{cases}
\Delta w_{\lambda}^{i}=0 & \text{in }\Omega-E_{\lambda},\\
w_{\lambda}^{i}=\varphi^{i} & \text{on }\partial\Omega,\\
w_{\lambda}^{i}=0 & \text{on }\partial E_{\lambda}.
\end{cases}
\]
Define $\mathbf{v}_{\lambda}(y):=\mathbf{w}_{\lambda}(T_{r,\lambda}^{-1}(y))$.
We are going to show that $\lambda\mapsto\mathbf{v}_{\lambda}$ is
a $C^{1}$ map from a neighborhood of $\lambda=0$ into $W^{1,2}(\Omega-E)$.
We know that each $v_{\lambda}^{i}$ satisfies an elliptic PDE of
the form 
\[
F[v,\lambda]=F(D_{y}^{2}v,\nabla_{y}v,y,\lambda)=0,\qquad\text{ in }\;U=\Omega-E.
\]
We also know that $F=\Delta$ when $y\notin B_{r}(x_{1})\cup B_{r}(x_{2})$
or when $\lambda=0$. In addition, we can consider $F$ as a $C^{1}$
map 
\begin{align*}
F:W^{1,2}(U)\times\mathbb{R} & \to W^{-1,2}(U),\\
(v,\lambda) & \mapsto F[v,\lambda]
\end{align*}
where $U=\Omega-E$. 

Now we employ the implicit function theorem to show that $\lambda\mapsto\mathbf{v}_{\lambda}$
is $C^{1}$. This can be readily deduced from the fact that 
\[
\partial_{v}F|_{\lambda=0}:W_{0}^{1,2}(U)\to W^{-1,2}(U)
\]
is invertible, since we have 
\[
\partial_{v}F|_{\lambda=0}\;\cdot=\frac{d}{ds}\Big|_{s=0}F[v+s\,\cdot\,,0]=\frac{d}{ds}\Big|_{s=0}\Delta(v+s\,\cdot\,)=\Delta\cdot.
\]
Therefore, $\mathbf{v}_{\lambda}=\mathbf{u}+\lambda\mathbf{u}_{0}+o(\lambda)$
in $W^{1,2}(U)$, where $\mathbf{u}_{0}\in W_{0}^{1,2}(U)$ solves
\[
0=\frac{d}{d\lambda}\Big|_{\lambda=0}F[v_{\lambda}^{i},\lambda]=\partial_{v}Fu_{0}^{i}+\partial_{\lambda}F.
\]
In other words 
\[
\Delta u_{0}^{i}=-\partial_{\lambda}F|_{v=u^{i},\,\lambda=0}.
\]
Note that we also have $\nabla\mathbf{v}_{\lambda}=\nabla\mathbf{u}+\lambda\nabla\mathbf{u}_{0}+o(\lambda)$,
since $\lambda\mapsto\mathbf{v}_{\lambda}$ is a $C^{1}$ map into
$W^{1,2}(U)$; so $\lambda\mapsto\nabla\mathbf{v}_{\lambda}$ is a
$C^{1}$ map into $L^{2}(U)$. 

Now let $h^{i}$ be the solution of $\Delta h^{i}=0$ in $U=\Omega-E$
with boundary data $h^{i}=\partial_{\xi_{i}}\Gamma(x,\partial_{\nu}\mathbf{u})$
on $\partial\Omega$ and $h^{i}=0$ on $\partial E$. Then for small
$\lambda>0$ we have (note that for $p=2$ we have $A_{\nu}=\partial_{\nu}$)
\begin{align*}
 & \int_{\partial\Omega}\Gamma(x,\partial_{\nu}\mathbf{v}_{\lambda})-\Gamma(x,\partial_{\nu}\mathbf{u})\,d\sigma\\
 & \hspace{1cm}=\int_{\partial\Omega}\sum_{i}\partial_{i}\Gamma(x,\partial_{\nu}\mathbf{u})(\partial_{\nu}v_{\lambda}^{i}-\partial_{\nu}u^{i})\,d\sigma+o(\lambda)\\
 & \hspace{1cm}=\lambda\int_{\partial\Omega}\sum_{i}\partial_{i}\Gamma(x,\partial_{\nu}\mathbf{u})\partial_{\nu}u_{0}^{i}\,d\sigma+o(\lambda)\\
 & \hspace{1cm}=\lambda\int_{\partial\Omega}\sum_{i}h^{i}\partial_{\nu}u_{0}^{i}\,d\sigma+o(\lambda)\\
 & \hspace{1cm}=\lambda\sum_{i}\int_{U}\nabla h^{i}\cdot\nabla u_{0}^{i}+h^{i}\Delta u_{0}^{i}\,dx+o(\lambda)\\
 & \hspace{1cm}=-\lambda\sum_{i}\int_{{\textstyle (}B_{r}(x_{1})\cup B_{r}(x_{2}){\textstyle )}-E}h^{i}\partial_{\lambda}F|_{v=u^{i},\,\lambda=0}\,dx+o(\lambda).
\end{align*}
Note that in the last line we have used the facts that $\Delta h^{i}=0$
in $U$ and $u_{0}^{i}=0$ on $\partial U=\partial\Omega\cup\partial E$.
Also, $\partial_{\lambda}F|_{v=u^{i},\,\lambda=0}=0$ outside $B_{r}(x_{1})\cup B_{r}(x_{2})$,
because in that region $F=\Delta$ for all $\lambda$.

Now let us extend $\mathbf{w}_{\lambda}$ to all of $\Omega$ by setting
it equal to $0$ on $E_{\lambda}$. Note that $w_{\lambda}^{i}$ is
positive on $\Omega-E_{\lambda}$ by the maximum principle. Hence
\[
\{|\mathbf{w}_{\lambda}|>0\}=\Omega-E_{\lambda}.
\]
Furthermore, similarly to the proof of Theorem \ref{thm: vol is 1},
we obtain 
\begin{align*}
f_{\varepsilon}\big(\big|\{|\mathbf{w}_{\lambda}|>0\}\big|\big) & -f_{\varepsilon}\big(\big|\{|\mathbf{u}|>0\}\big|\big)\le\frac{1}{\varepsilon}(|E|-|E_{\lambda}|)\\
 & =\frac{\lambda}{\varepsilon}\Big(\int_{B_{r}(x_{2})\cap\{|\mathbf{u}|>0\}}\rho'(|x-x_{2}|)\frac{\langle x-x_{2},\nu(x_{2})\rangle}{|x-x_{2}|}\,dx\\
 & \qquad\;-\int_{B_{r}(x_{1})\cap\{|\mathbf{u}|>0\}}\rho'(|x-x_{1}|)\frac{\langle x-x_{1},\nu(x_{1})\rangle}{|x-x_{1}|}\,dx\Big)=\frac{\lambda}{\varepsilon}o(r^{n}).
\end{align*}
Therefore if we compare the energy of $\mathbf{u}$ with $\mathbf{w}_{\lambda}$
(it is easy to see that $\mathbf{w}_{\lambda}\in V$) we get (in the
second equality below we use the fact that $\mathbf{v}_{\lambda}=\mathbf{w}_{\lambda}$
near $\partial\Omega$) 
\begin{align*}
0\le J_{\varepsilon}(\mathbf{w}_{\lambda})-J_{\varepsilon}(\mathbf{u}) & =\int_{\partial\Omega}\Gamma(x,\partial_{\nu}\mathbf{w}_{\lambda})-\Gamma(x,\partial_{\nu}\mathbf{u})\,d\sigma\\
 & \qquad\qquad+f_{\varepsilon}\big(\big|\{|\mathbf{w}_{\lambda}|>0\}\big|\big)-f_{\varepsilon}\big(\big|\{|\mathbf{u}|>0\}\big|\big)\\
 & =\int_{\partial\Omega}\Gamma(x,\partial_{\nu}\mathbf{v}_{\lambda})-\Gamma(x,\partial_{\nu}\mathbf{u})\,d\sigma+\frac{\lambda}{\varepsilon}o(r^{n})\\
 & =-\lambda\sum_{i}\int_{{\textstyle (}B_{r}(x_{1})\cup B_{r}(x_{2}){\textstyle )}-E}h^{i}\partial_{\lambda}F|_{v=u^{i},\,\lambda=0}\,dx+o(\lambda)+\frac{\lambda}{\varepsilon}o(r^{n}).
\end{align*}
Hence if we divide by $\lambda$ and let $\lambda\to0$ we obtain
\begin{equation}
0\le-\sum_{i}\int_{{\textstyle (}B_{r}(x_{1})\cup B_{r}(x_{2}){\textstyle )}-E}h^{i}\partial_{\lambda}F|_{v=u^{i},\,\lambda=0}\,dx+o(r^{n}).\label{Compar:v-u}
\end{equation}
So we need to compute $\partial_{\lambda}F|_{v=u^{i},\,\lambda=0}$. 

Next let us compute $F$ explicitly. Set $x=T_{r,\lambda}^{-1}(y)$
so that $y=T_{r,\lambda}(x)$. To simplify the notation we suppress
the $\lambda$ or $r$ in the indices. We have $v^{i}(T(x))=v^{i}(y)=w^{i}(x)$.
Hence 
\begin{align*}
\partial_{x_{k}}w^{i} & =\sum_{j}\partial_{y_{j}}v^{i}\partial_{x_{k}}T^{j},\\
\partial_{x_{k}x_{k}}^{2}w^{i} & =\sum_{j}\partial_{x_{k}}\big(\partial_{y_{j}}v^{i}\partial_{x_{k}}T^{j}\big)\\
 & =\sum_{j,\ell}\partial_{y_{j}y_{\ell}}^{2}v^{i}\partial_{x_{k}}T^{j}\partial_{x_{k}}T^{\ell}+\sum_{j}\partial_{y_{j}}v^{i}\partial_{x_{k}x_{k}}^{2}T^{j}.
\end{align*}
Therefore 
\[
0=\Delta w^{i}=\sum_{j,\ell,k}\partial_{y_{j}y_{\ell}}^{2}v^{i}\partial_{x_{k}}T^{j}\partial_{x_{k}}T^{\ell}+\sum_{j,k}\partial_{y_{j}}v^{i}\partial_{x_{k}x_{k}}^{2}T^{j}.
\]
It is easy to see that inside $B_{r}(x_{a})$ ($a=1,2$) we have 
\begin{align*}
\partial_{x_{k}}T^{j} & =\delta_{jk}+(-1)^{a}\lambda\rho'(|z|)\frac{z_{k}}{|z|}\nu^{j}(x_{a}),\qquad\big(z=\frac{x-x_{a}}{r}\big),\\
\partial_{x_{k}x_{k}}^{2}T^{j} & =(-1)^{a}\lambda\partial_{x_{k}}\big(\rho'(|z|)\frac{z_{k}}{|z|}\big)\nu^{j}(x_{a}).
\end{align*}
Thus 
\begin{align*}
F[v,\lambda] & =\sum_{j,\ell,k}\partial_{y_{j}y_{\ell}}^{2}v\partial_{x_{k}}T^{j}\partial_{x_{k}}T^{\ell}+\sum_{j,k}\partial_{y_{j}}v\partial_{x_{k}x_{k}}^{2}T^{j}\\
 & =\sum_{j,\ell,k}\big[\delta_{jk}+(-1)^{a}\lambda\rho'(|z|)\frac{z_{k}}{|z|}\nu^{j}(x_{a})\big]\big[\delta_{\ell k}+(-1)^{a}\lambda\rho'(|z|)\frac{z_{k}}{|z|}\nu^{\ell}(x_{a})\big]\partial_{y_{j}y_{\ell}}^{2}v\\
 & \qquad\qquad+\sum_{j,k}\big[(-1)^{a}\lambda\partial_{x_{k}}\big(\rho'(|z|)\frac{z_{k}}{|z|}\big)\nu^{j}(x_{a})\big]\partial_{y_{j}}v
\end{align*}
in $B_{r}(x_{a})$ for $a=1,2$, and $F[v,\lambda]=\Delta v$ elsewhere.
Now note that 
\[
\sum_{k}\partial_{x_{k}}\big(\rho'(|z|)\frac{z_{k}}{|z|}\big)=\sum_{k}\big(\rho''(|z|)\frac{z_{k}^{2}}{r|z|^{2}}+\rho'(|z|)\frac{1}{r|z|}-\rho'(|z|)\frac{z_{k}^{2}}{r|z|^{3}}\big)=\frac{1}{r}\rho''(|z|).
\]
Hence we get 
\[
\partial_{\lambda}F|_{v=u^{i},\,\lambda=0}=(-1)^{a}\Big(2\rho'(|z|)\sum_{j,k}\frac{z_{k}}{|z|}\nu^{j}(x_{a})\partial_{jk}^{2}u^{i}+\frac{1}{r}\rho''(|z|)\sum_{j}\nu^{j}(x_{a})\partial_{j}u^{i}\Big)
\]
in $B_{r}(x_{a})$ for $a=1,2$. Note that although a priori $z,u^{i}$
in the above equation are functions of $y$, at $\lambda=0$ we have
$y=x$, and thus we can regard them as functions of $x$ too.

Let $\mathbf{u}_{r}(z)=\frac{1}{r}\mathbf{u}(x_{a}+rz)=\frac{1}{r}\mathbf{u}(x)$
and $h_{r}^{i}(z)=\frac{1}{r}h^{i}(x_{a}+rz)=\frac{1}{r}h^{i}(x)$.
Putting all these in (\ref{Compar:v-u}) we get (note that in the following integration by parts the boundary term is zero, since $\rho$ is $0$ for $z$ near $\partial B_{1}$ and $h^{i}$ is $0$ on $\partial E$) 
\begin{align*}
0 & \le-\sum_{i}\int_{{\textstyle (}B_{r}(x_{1})\cup B_{r}(x_{2}){\textstyle )}-E}h^{i}\partial_{\lambda}F|_{v=u^{i},\,\lambda=0}\,dx+o(r^{n})\\
 & =\sum_{a,i}(-1)^{a+1}\int_{B_{r}(x_{a})-E}h^{i}\Big(2\rho'(|z|)\sum_{j,k}\frac{z_{k}}{|z|}\nu^{j}(x_{a})\partial_{jk}^{2}u^{i}\\
 & \qquad\qquad\qquad\qquad\qquad\qquad\qquad\qquad+\frac{1}{r}\rho''(|z|)\sum_{j}\nu^{j}(x_{a})\partial_{j}u^{i}\Big)dx+o(r^{n})\allowdisplaybreaks\\
 & =\sum_{a,i}(-1)^{a+1}\int_{B_{r}(x_{a})-E}\Bigl(-2\sum_{k}\partial_{k}\Big[h^{i}\rho'(|z|)\frac{z_{k}}{|z|}\Big]\sum_{j}\nu^{j}(x_{a})\partial_{j}u^{i}\\
 & \qquad\qquad\qquad\qquad\qquad\qquad\qquad\qquad+\frac{1}{r}h^{i}\rho''(|z|)\sum_{j}\nu^{j}(x_{a})\partial_{j}u^{i}\Big)dx+o(r^{n})\allowdisplaybreaks\\
 & =\sum_{a,i}(-1)^{a+1}\int_{B_{r}(x_{a})-E}\Bigl(-2\sum_{k}\Big[\partial_{k}h^{i}\rho'(|z|)\frac{z_{k}}{|z|}+h^{i}\partial_{k}\big(\rho'(|z|)\frac{z_{k}}{|z|}\big)\Big]\\
 & \qquad\qquad\qquad\qquad\qquad\qquad\qquad\qquad+\frac{1}{r}h^{i}\rho''(|z|)\Big)\sum_{j}\nu^{j}(x_{a})\partial_{j}u^{i}\,dx+o(r^{n})\allowdisplaybreaks\\
 & =\sum_{a,i}(-1)^{a+1}\int_{B_{r}(x_{a})-E}\Bigl(-2\sum_{k}\Big[\partial_{k}h^{i}\rho'(|z|)\frac{z_{k}}{|z|}\Big]\\
 & \qquad\qquad\qquad\qquad\qquad\qquad\qquad\qquad-\frac{1}{r}h^{i}\rho''(|z|)\Big)\sum_{j}\nu^{j}(x_{a})\partial_{j}u^{i}\,dx+o(r^{n})\allowdisplaybreaks\\
 & =\sum_{a,i}(-1)^{a+1}r^{n}\int_{B_{1}\cap\{|\mathbf{u}_{r}|>0\}}\Bigl(-2\sum_{k}\Big[\partial_{k}h_{r}^{i}\rho'(|z|)\frac{z_{k}}{|z|}\Big]\\
 & \qquad\qquad\qquad\qquad\qquad\qquad\qquad\qquad-\frac{1}{r}rh_{r}^{i}\rho''(|z|)\Big)\sum_{j}\nu^{j}(x_{a})\partial_{j}u_{r}^{i}\,dz+o(r^{n}).
\end{align*}
Now note that $\partial_{j}u_{r}^{i}(z)\to q^{i}(x_{a})\nu^{j}(x_{a})=\partial_{j}u^{i}(x_{a})$
when $z\cdot\nu(x_{a})>0$ by the results of \cite{alt1984free}.
Next note that $h^{i}$ is Lipschitz continuous, since $u^{i}$ is Lipschitz and we have $0\le h^{i}\le cu^{i}$
for some constant $c$. To see this note that the function $\partial_{\xi_{i}}\Gamma(x,\partial_{\nu}\mathbf{u})$
is positive and continuous on the compact set $\partial\Omega$, so
it is bounded there, and thus for some $c>0$ we have $h^{i}=\partial_{\xi_{i}}\Gamma(x,\partial_{\nu}\mathbf{u})\le c\varphi^{i}=cu^{i}$
on $\partial\Omega$. Hence the claim follows by the maximum principle.
Therefore, by Lemma B.1 in \cite{fotouhi2023minimization}, we also
have $\partial_{k}h_{r}^{i}(z)\to p^{i}(x_{a})\nu^{k}(x_{a})=\partial_{k}h^{i}(x_{a})$
for some function $p^{i}$, and $h_{r}^{i}(z)\to\nabla h^{i}(x_{a})\cdot z$
as $h^{i}(x_{a})=0$. Thus if we divide the above expression by $r^{n}$
and let $r\to0$ we obtain 
\begin{align*}
0 & \le\sum_{a,i}(-1)^{a+1}\int_{B_{1}\cap\{z\cdot\nu(x_{a})>0\}}\Bigl(-2\sum_{k}\Big[\partial_{k}h^{i}(x_{a})\rho'(|z|)\frac{z_{k}}{|z|}\Big]\\
 & \qquad\qquad\qquad\qquad\qquad\qquad\qquad-\big(\nabla h^{i}(x_{a})\cdot z\big)\rho''(|z|)\Big)\sum_{j}\nu^{j}(x_{a})\partial_{j}u^{i}(x_{a})\,dz\allowdisplaybreaks\\
 & =\sum_{a,i}(-1)^{a+1}\int_{B_{1}\cap\{z\cdot\nu(x_{a})>0\}}\Bigl(-2\sum_{k}\Big[p^{i}(x_{a})\nu^{k}(x_{a})\rho'(|z|)\frac{z_{k}}{|z|}\Big]\\
 & \qquad\qquad\qquad\qquad\qquad\qquad\qquad\qquad-p^{i}(x_{a})\big(\nu(x_{a})\cdot z\big)\rho''(|z|)\Big)\partial_{\nu}u^{i}(x_{a})\,dz\allowdisplaybreaks\\
 & =\sum_{a,i}(-1)^{a}\int_{B_{1}\cap\{z\cdot\nu(x_{a})>0\}}\Bigl(\frac{2}{|z|}\rho'(|z|)+\rho''(|z|)\Big)\big(\nu(x_{a})\cdot z\big)p^{i}(x_{a})\partial_{\nu}u^{i}(x_{a})\,dz\\
 & =\sum_{a,i}(-1)^{a}\partial_{\nu}h^{i}(x_{a})\partial_{\nu}u^{i}(x_{a})\int_{B_{1}\cap\{z\cdot\nu(x_{a})>0\}}\Bigl(\frac{2}{|z|}\rho'(|z|)+\rho''(|z|)\Big)\big(\nu(x_{a})\cdot z\big)dz\\
 & =C_{\rho}\Big(\sum_{i}\partial_{\nu}h^{i}(x_{2})\partial_{\nu}u^{i}(x_{2})-\sum_{i}\partial_{\nu}h^{i}(x_{1})\partial_{\nu}u^{i}(x_{1})\Big),
\end{align*}
where $C_{\rho}=\int_{B_{1}\cap\{z\cdot\nu(x_{a})>0\}}\bigl(\frac{2}{|z|}\rho'(|z|)+\rho''(|z|)\big)\big(\nu(x_{a})\cdot z\big)dz$
does not depend on $x_{a}$; we have also used the fact that $p^{i}(x_{a})=\partial_{\nu}h^{i}(x_{a})$.
By switching the role of $x_{1},x_{2}$ we conclude that 
\[
\sum_{i}\partial_{\nu}h^{i}(x_{2})\partial_{\nu}u^{i}(x_{2})-\sum_{i}\partial_{\nu}h^{i}(x_{1})\partial_{\nu}u^{i}(x_{1})
\]
must be zero, as desired.
\end{proof}

The main idea to show the regularity of the free boundary lies in
utilizing the boundary Harnack principle, which allows us to reduce the system into a scalar problem. The key tool in employing this approach is non-tangential accessibility of the domain; for the definition of non-tangentially accessible (NTA) domains we refer to \cite{aguilera1987optimization}.

\begin{lem}
Let $\mathbf{u}$ be a solution of the minimization problem (\ref{eq: main_eq})
for $p=2$. Then $U=\{x:|\mathbf{u}(x)|>0\}$ is a non-tangentially
accessible domain. 
\end{lem}

\begin{proof}
This result follows from the same analysis as of Theorem 4.8 in \cite{aguilera1987optimization}
for the function ${\scriptstyle \mathcal{U}}=u^{1}+\cdots+u^{m}$.
Note that ${\scriptstyle \mathcal{U}}$ is harmonic in $\{|\mathbf{u}|>0\}=\{{\scriptstyle \mathcal{U}}>0\}$
(these two sets are equal due to Lemma \ref{lem:coperative}), and the
function ${\scriptstyle \mathcal{U}}$ is also Lipschitz continuous
and satisfies the nondegeneracy property by Corollary \ref{cor-nondegeneracy}.
\end{proof}

\begin{thm}
Let $x_{0}\in\mathcal{R}$ be a regular point of the free boundary.
Then there is $r> 0$
such that $B_r(x_{0})\cap\partial\{|\mathbf{u}|>0\}$  is a $C^{1,\alpha}$ hypersurface for some $\alpha>0$. 
\end{thm}

\begin{proof}
We may assume that $u^{1}>0$ in $B_{r_{0}}(x_{0})\cap\{|\mathbf{u}|>0\}$ for some $r_0>0$.
First we show that for some $0<r\le r_{0}$ there is a H\"older continuous
function $g$ defined on $B_{r}(x_{0})\cap\partial\{|\mathbf{u}|>0\}$,
such that in the viscosity sense we have 
\[
\partial_{\nu}h^{1}\partial_{\nu}u^{1}=g\quad\text{ on }\partial E,
\]
where $h^{1}$ is defined in Lemma \ref{lem:FBC}. Since $B_{r_{0}}(x_{0})\cap\{|\mathbf{u}|>0\}$
is an NTA domain, the boundary Harnack inequality implies that $G^{i}:=u^{i}/u^{1}$
and $H^{i}=h^{i}/h^{1}$ are H\"older continuous functions in $B_{r}(x_{0})\cap\overline{\{|\mathbf{u}|>0\}}$
for some $0<r\le r_{0}$. Now if we consider a one-sided tangent ball
at some point $y\in B_{r}(x_{0})\cap\partial\{|\mathbf{u}|>0\}$,
we have asymptotic developments (see Lemma B.1 in \cite{fotouhi2023minimization},
noting that $h^{i}$ is Lipschitz as we have shown in the proof of
Lemma \ref{lem:FBC}) 
\[
\begin{split}u^{i}(y+x) & =q^{i}(y)(x\cdot\nu(y))^{+}+o(|x|),\\
h^{i}(y+x) & =p^{i}(y)(x\cdot\nu(y))^{+}+o(|x|).
\end{split}
\]
Therefore $G^{i}(y)=q^{i}(y)/q^{1}(y)$ and $H^{i}(y)=p^{i}(y)/p^{1}(y)$.
Thus from \eqref{FB-condition} we can infer that 
\[
p^{1}(y)q^{1}(y)\Big(1+\sum_{i>1}G^{i}(y)H^{i}(y)\Big)=\sum_{i}p^{i}(y)q^{i}(y)=\sum_{i}\partial_{\nu}h^{i}\partial_{\nu}u^{i}
\]
is constant for every $y\in B_{r}(x_{0})\cap\partial\{|\mathbf{u}|>0\}$. Note that $G^{i},H^{i}>0$ at $y$ as $p^{i},q^{i}>0$. Hence by applying Theorem 3.1 in \cite{maiale2021epsilon}
we get the desired result. 
\end{proof}

\begin{cor}
Let $\mathbf{u}$ be a solution of the minimization problem (\ref{eq: main_eq})
for $p=2$. Then the regular part of the free boundary, $\mathcal{R}$,
is analytic. 
\end{cor}

\begin{proof}
Suppose $0\in\mathcal{R}$ and $u^{1}>0$ in $B_{r}\cap\{|\mathbf{u}|>0\}$.
Then we apply the hodograph-Legendre transformation $x\mapsto y=(x_{1},\dots,x_{n-1},u^{1})$.
Next we define the partial Legendre functions 
\begin{align*}
v^{1}(y):=x_{n},\quad v^{i}(y):=u^{i}(x), & \quad\text{for }i=2,\cdots,m,\\
w^{i}(y):=h^{i}(x), & \quad\text{for }i=1,\cdots,m.
\end{align*}
As $\mathcal{R}$ is $C^{1,\alpha}$, it follows that $u^{i}$ and
$h^{i}$ are in $C^{1,\alpha}(\overline{B_{r}\cap\{|\mathbf{u}|>0\}})$.
So, $v^{i}$ and $w^{i}$ are $C^{1,\alpha}$ in a neighborhood of
the origin in $\{y_{n}\ge0\}$. Now we have verified all the hypothesis
of Theorem 7.1 in \cite{aguilera1987optimization}, and through a similar argument we can obtain the analyticity of $\mathcal{R}$.
\end{proof}

 \section*{Acknowledgements}
This project was initiated  while the authors stayed at Institute Mittag Leffler (Sweden), during the program Geometric aspects of nonlinear PDE.
H. Shahgholian was supported by Swedish research Council.

\section*{Declarations}

\noindent {\bf  Data availability statement:} All data needed are contained in the manuscript.

\medskip

\noindent {\bf  Funding and/or Conflicts of interests/Competing interests:} The authors declare that there are no financial, competing or conflict of interests.

\bibliographystyle{plain}
\bibliography{Bibliography}

\end{document}